\documentclass[11pt, leqno]{amsart}
\usepackage[T1]{fontenc}
\usepackage[utf8]{inputenc}
\usepackage{amsfonts,amssymb, amscd}
\usepackage{esint}
\usepackage{amsmath,accents}
\usepackage{lmodern}
\usepackage{mathrsfs}
\usepackage{amsthm}
\usepackage{qsymbols}
\usepackage{mathtools}
\mathtoolsset{showonlyrefs}
\usepackage{dsfont}
\usepackage{graphicx}
\usepackage{latexsym}
\usepackage{chngcntr}
\usepackage[noadjust]{cite}
\usepackage{paralist}
\usepackage[parfill]{parskip}
\usepackage{bm}
\usepackage{tikz-cd}
\usepackage{nicefrac}
\usepackage{float}
\usepackage{csquotes}

\usepackage[shortlabels]{enumitem}
\setlist[enumerate,1]{label = \normalfont(\arabic*), ref = (\arabic*)}
\usepackage{euflag}

\usepackage{todonotes}
\newtheorem{theorem}{Theorem}[section]
\newtheorem{lemma}[theorem]{Lemma}
\newtheorem{proposition}[theorem]{Proposition}
\newtheorem{corollary}[theorem]{Corollary}
\newtheorem{taggedtheoremx}{Assumption}
\newenvironment{assumption}[1]
{\renewcommand\thetaggedtheoremx{#1}\taggedtheoremx}
{\endtaggedtheoremx}
\theoremstyle{definition}
\newtheorem{definition}[theorem]{Definition}
\newtheorem{remark}[theorem]{Remark}

\newtheorem{convention}[theorem]{Convention}

\usepackage[plainpages=false,pdfpagelabels,
citecolor=red]{hyperref}
\usepackage{xcolor}
\hypersetup{
	colorlinks,
	linkcolor={cyan!90!black},
	citecolor={magenta},
	urlcolor={green!40!black}
} %

\newcommand{\R}{\mathbb{R}}
\newcommand{\C}{\mathbb{C}}
\newcommand{\N}{\mathbb{N}}
\renewcommand{\L}{\mathrm{L}}
\newcommand{\W}{\mathrm{W}}
\newcommand{\IW}{\mathbb{W}}
\newcommand{\Cont}{\mathrm{C}}

\newcommand{\e}{\mathrm{e}}
\renewcommand{\d}{\,\mathrm{d}}
\newcommand{\ddt}{\,\frac{\mathrm{d}t}{t}}

\newcommand{\rad}{\mathrm{r}}

\newcommand{\eps}{\varepsilon}
\newcommand{\B}{\mathrm{B}}
\newcommand{\Q}{\mathrm{Q}}
\renewcommand{\H}{\mathrm{H}}
\newcommand{\cH}{\mathcal{H}}
\newcommand{\A}{\mathcal{A}}
\newcommand{\cl}[1]{\overline{#1}}

\renewcommand\Re{\operatorname{Re}}
\renewcommand\Im{\operatorname{Im}}

\DeclareMathOperator{\bd}{\partial \!}
\DeclareMathOperator{\supp}{supp}
\DeclareMathOperator{\dist}{d}
\DeclareMathOperator{\diam}{diam}
\DeclareMathOperator{\dom}{D}
\DeclareMathOperator{\Div}{div}
\newcommand{\Sec}{\mathrm{S}}
\newcommand{\cE}{\mathcal{E}}
\newcommand{\bD}{{\boldsymbol{D}}}
\newcommand{\bO}{{\boldsymbol{O}}}
\newcommand{\bL}{\boldsymbol{L}}
\newcommand{\bTheta}{\boldsymbol{\Theta}}
\newcommand{\bP}{\boldsymbol{P}}
\newcommand{\bQ}{\boldsymbol{Q}}
\newcommand{\bcL}{\boldsymbol{\mathcal{L}}}
\newcommand{\bgamma}{\boldsymbol{\gamma}}
\newcommand{\bB}{\boldsymbol{B}}
\newcommand{\Res}{\mathcal{R}}

\DeclareRobustCommand{\Ar}[1]{{\overrightarrow{#1}}}
\newcommand{\bAr}[1]{{\overrightarrow{\boldsymbol{#1}}}}

\DeclareMathOperator{\ran}{ran}
\DeclareMathOperator{\loc}{loc}

\newcommand{\one}{\mathbf{1}}

\newcommand{\les}{\lesssim}

\newcommand{\norm}[1]{\left\lVert#1\right\rVert}
\newcommand{\abs}[1]{\left\lvert#1\right\rvert}
\newcommand{\normm}[1]{\lVert#1\rVert}
\newcommand{\abss}[1]{\lvert#1\rvert}
\newcommand{\scal}[1]{\left\langle#1\right\rangle}

\makeatletter
\providecommand{\leftsquigarrow}{%
  \mathrel{\mathpalette\reflect@squig\relax}%
}
\newcommand{\reflect@squig}[2]{%
  \reflectbox{$\m@th#1\rightsquigarrow$}%
}
\makeatother

\newcommand{\embeds}{\hookrightarrow}

\makeatletter
\@namedef{subjclassname@2020}{%
	\textup{2020} Mathematics Subject Classification}
\makeatother

\title[A second order approach to the Kato square root problem on open sets]{A second order approach to \\ the Kato square root problem on open sets}

\author[S. Bechtel]{Sebastian Bechtel}
\address{Université Paris-Saclay, CNRS \\ Laboratoire de Mathématiques d’Orsay \\ 91405 Orsay \\ France}
\email{sebastian.bechtel@universite-paris-saclay.fr}

\author[C. Hutcheson]{Cody Hutcheson}
\address{Department of Mathematics, University of Alabama, Tuscaloosa, AL, 35487, USA}
\email{cmhutcheson@crimson.ua.edu}

\author[T. Schmatzler]{Timotheus Schmatzler}
\address{Matematiska institutionen, Stockholms universitet, 106 91 Stockholm, Sweden}
\email{timotheus.schmatzler@math.su.se}

\author[T. Tasci]{Tolgahan Tasci}
\address{IMACM, Bergische Universität Wuppertal, Gaußstraße 20, Wuppertal, D-42119, Germany}
\email{tasci@uni-wuppertal.de}

\author[M. Wittig]{Mattes Wittig}
\address{Maths Institute, TU Hamburg (TUHH), D-21073 Hamburg, Germany}
\email{mattes.wittig@tuhh.de}

\date{\today}
\thanks{The first-named author was supported by the Alexander von Humboldt foundation by a Feodor Lynen grant. The third-named author gratefully acknowledges financial support by the grant no. 2022-03342 of the Swedish Research Council (VR). This project has received funding from the European Union’s Horizon 2020 research and innovation programme under the Marie Skłodowska-Curie grant agreement No 101034255 \euflag{}}
\dedicatory{}
\begin{document}
	\begin{abstract}

    We obtain the Kato square root property for coupled second-order elliptic systems in divergence form subject to mixed boundary conditions on an open and possibly unbounded set in $\mathbb{R}^n$ under two simple geometric conditions: The Dirichlet boundary parts for the respective components are Ahlfors--David regular and a quantitative connectivity property in the spirit of locally uniform domains holds near the remaining Neumann boundary parts. In contrast to earlier work, our proof is not based on the first-order approach due to Axelsson--Keith--McIntosh but uses a second-order approach in the spirit of the original solution to the Kato square root problem on Euclidean space. This way, the proof becomes substantially shorter and technically less demanding.

	\end{abstract}
	\maketitle

	\section{Introduction}
	\label{sec:intro}

    Let $L$ be a coupled second-order elliptic system in divergence form on an open set $O \subseteq \R^n$, $n\geq 2$, with bounded, measurable, elliptic coefficients in $\mathcal{L}(\C^m)$, formally given by
    \begin{align}
        Lu = -\sum_{i,j=1}^n \partial_i (a_{ij} \partial_j u) - \sum_{i=1}^n \partial_i (b_i u) + \sum_{j=1}^n c_i \partial_i u + d u.
    \end{align}
    The number $m$ is the system size. Precise assumptions on the coefficients are presented in Assumption~\ref{ass: coeffs of L}. For each component $i = 1, \dots, m$, let $D_i \subseteq \partial O$ be a closed, possibly empty, subset of the boundary. Group them in the array $\Ar{D} = (D_i)_{i=1}^m$. As an abuse of notation, we also identify $\Ar{D}$ with a subset of $(\partial O)^m$. Then, we complement the system $L$ with Dirichlet boundary conditions on $\Ar{D}$. On the remaining (Neumann) boundary $(\partial O)^m \setminus \Ar{D}$, natural boundary conditions are imposed.

    To make the definition precise, define $\W^{1,2}_{D_i}(O)$ as the closure in $\W^{1,2}(O)$ of test functions vanishing in a neighborhood of $D_i$ (see Definition~\ref{Def: W12D}), and put $\IW^{1,2}_\Ar{D}(O) = \otimes_{i=1}^m \W^{1,2}_{D_i}(O)$.
    The system $L$ can be properly defined using a sesquilinear form $$\mathfrak{a} \colon \IW^{1,2}_\Ar{D}(O) \times \IW^{1,2}_\Ar{D}(O) \to \C$$ using Kato's form method. The form $\mathfrak{a}$ is obtained from the (formal) definition of $L$ above by virtue of integration by parts, see~\eqref{eq: def of sesqform a}.

    The system $L$ is invertible and $m$-accretive in $\L^2(O)^m$. In particular, it possesses a (unique) $m$-accretive square root $\sqrt{L}$ with the property $(\sqrt{L})^2 = L$. The question whether or not the identity $\dom(\sqrt{L}) = \IW^{1,2}_\Ar{D}(O)$ holds has become famous as the \emph{Kato square root problem}. Let us mention that, if the coefficients of the system are self-adjoint, then this property is essentially trivial. However, in the non-symmetric case, even the case $O = \R^n$ and $m=1$ has resisted all attempts of a resolution for more than 40 years. Eventually, it was the breakthrough result due to Auscher, Hofmann, Lacey, McIntosh, and Tchamitchian~\cite{Kato-Square-Root-Proof} that answered Kato's question in the affirmative in the Euclidean case.

    Given that $L$ is an elliptic operator of second order, it might not seem too surprising that the proof in~\cite{Kato-Square-Root-Proof} relies on harmonic analysis techniques for second-order elliptic operators. Now, however, comes the plot twist: few years after the original solution of the Kato square root problem, Axelsson, Keith, and McIntosh proposed in their seminal paper~\cite{AKM-QuadraticEstimates} a new framework, called the \emph{first-order approach}, in which the Kato square root problem (and many other problems from the so-called Calderón program) could be casted by reformulating them in terms of certain first-order differential operators (the so-called \emph{perturbed Dirac operators}). Moreover, using this approach, they were also capable of resolving first cases of the Kato square root problem when the operator $L$ is subject to mixed boundary conditions~\cite{AKM}.
    Eventually, using the first-order approach became the gold standard when treating the Kato square root problem on manifolds~\cite{Morris} or rough open sets subject to mixed boundary conditions~\cite{BEH-Kato,Darmstadt-KatoMixedBoundary,Laplace-Extrapolation}.

    Applications on sets with rough geometry come from various fields including, with exemplary references, elliptic and parabolic regularity~\cite{RobertJDE, Bonifacius-Neitzel}, Lions' non-autonomous maximal regularity problem~\cite{Achache-Ouhabaz,Fackler} and boundary value problems~\cite{AAM-ArkMath, AE-mixed}.

    The reader might wonder why the second-order proof on $\R^n$ does \emph{not} generalize to the setting on an open set $O$. To give an explanation, we have to dive a bit into the structure of the proof. What is common in both the second and first-order approaches is that the original problem is reduced to the validity of some quadratic estimate. In our case, the quadratic estimate reads
    \begin{align}
    \label{eq:intro:QE}
    \tag{$\heartsuit$}
        \int_0^\infty \norm{ \Theta_t \begin{bmatrix}
            u \\ \nabla u
        \end{bmatrix}}^2_2
        \ddt \les \norm{ u }^2_{\IW^{1,2}(O)} \, ,
        \qquad u \in \IW^{1,2}_\Ar{D} (O) \, .
    \end{align}
    The operator $\Theta_t$ transforms a \enquote{vector-field} in $\C^{(n+1)m}$ into a vector in $\C^m$;
    see Section~\ref{sec:reduction} for a precise definition of $\Theta_t$ and more details. To obtain the quadratic estimate~\eqref{eq:intro:QE}, the task is further reduced to a Carleson measure estimate. The latter is the heart of the matter and relies on involved local $T(b)$--techniques, see Section~\ref{sec:carleson}. The bottleneck for the treatment of open sets is, however, the reduction to a Carleson measure estimate, presented in Section~\ref{sec:principal part approximation}.
    Indeed, to obtain the desired estimates, suitable smoothing operators in the spirit of Littlewood--Paley theory have to be introduced. In the original proof, these smoothing operators are convolution operators with a nice kernel which are applied component-wise to the vector field $\nabla u$, see~\cite[Lem.~4.4]{Kato-Square-Root-Proof}. In the same lemma, a crucial argument is performed: using the Fourier transform, these component-wise smoothing operators commute nevertheless with any partial derivatives. Also in modern treatments of the Kato square root problem on Euclidean space, the same argument is crucially invoked~\cite[Rem.~13.8]{ISEM}.
    Such a reasoning is, of course, impossible in our setting, nourishing the need for the first-order approach.

    The objective of the current manuscript is hence the following: we aim to reprove the state-of-the-art result~\cite{BEH-Kato} on the Kato square root problem subject to mixed boundary conditions by virtue of a second-order approach, that is, with a proof in the spirit of~\cite{Kato-Square-Root-Proof,ISEM}. Moreover, in~\cite{BEH-Kato} only the case $D_i \equiv D$ was treated. We are going to generalize~\cite{BEH-Kato} to allow a varying Dirichlet part among the different components, which was already possible in~\cite{Darmstadt-KatoMixedBoundary} in a more restrictive geometric setting.

    Our main result reads as follows.

    \begin{theorem}[Solution of the Kato square root problem]
		\label{thm:main}
        Let $O \subseteq \R^n$ be an open set and, for $i = 1, \dots, m$, let $D_i \subseteq \partial O$ be a closed subset of the boundary. Set $\Ar{D} = (D_i)_{i=1}^m$. Suppose for every $i$ that $D_i$ is Ahlfors--David regular (Definition~\ref{Def: AD}) and that $O$ is locally uniform near $\partial O \setminus D_i$ (Definition~\ref{def: locally eps-delta}). Then, $\dom(\sqrt{L}) = \IW^{1,2}_\Ar{D}(O)$ holds with equivalence of norms
        \begin{align}
            \norm{u}_2 + \norm{\nabla u}_2 \approx \| \sqrt{L} u \|_2 \mathrlap{\qquad (u \in \IW^{1,2}_\Ar{D}(O)),}
        \end{align}
        and the implicit constants depend on the coefficients of $L$ only through the coefficient bounds.
    \end{theorem}

    The geometric setting announced in the theorem is very general and includes settings way below the Lipschitz class, including fractals like the von Koch snowflake. We will give more information on the geometrical setting in Section~\ref{sec:geometry}.

    Before we explain the central idea of our proof, let us comment on the motivation for a second-order proof for Theorem~\ref{thm:main}. The first-order approach introduces an extra layer of information and complexity, which is crucial for applications to related areas such as elliptic boundary value problems (see the recent monograph~\cite{Block} for background), but can also impede extensions to other settings and questions. One example is the Kato square root problem for parabolic operators. First, it was solved by Auscher, Egert and Nyström~\cite{Parabolic_FirstOrder} using the first-order approach. Their proof is incredibly technical. Later, Ataei, Egert and Nyström~\cite{Parabolic_SecondOrder} gave another proof using the second-order approach, which is not only substantially shorter and easier, but also generalizes the previous result in several aspects. So far, the Kato square root problem for parabolic operators on cylindrical domains was only treated by Ouhabaz~\cite{Parabolic_Ouhabaz} using regularity in time of the coefficients. We expect that the methods developed in this article will prove useful in the treatment of the Kato square root problem for parabolic operators on cylindrical domains without any regularity of the coefficients, taking~\cite{Parabolic_SecondOrder} as a blueprint. Some of the authors of this article will pursue this question in future work.

    We would also like to mention the recent result of Haardt and Tolksdorf~\cite{Luca_Patrick_Kato} on the Kato square root problem for generalized Stokes operators, which employs likewise the second-order approach instead of the first-order approach due to reduced complexity.

    Now, we shed some light on the central insight underlying this article. As we have already mentioned, the standard proofs of the Kato square root problem using the second-order approach use smoothing operators that belong to the functional calculus of $-\Delta$ (or, if $L$ contains lower-order terms as in our case, $-\Delta + 1$ ) and apply them component-wise to the vector-field $\nabla u$ (or, with lower-order terms, to $(u, \nabla u)^t$). For simplicity, we stick to the inhomogeneous case with lower-order terms from now on. Write $Tu = (u, \nabla u)^t$. Then, at least formally, $-\Delta + 1 = T^* T$. Now, reversing the order of terms on the right-hand side of the last identity, we define the operator $M \coloneqq T T^*$. Similar to $-\Delta + 1$, the operator $M$ can be defined using a sesquilinear form, and shares many desirable properties such as self-adjointness, $m$-accretivity, the square root property (for self-adjoint operators) and so forth with it. We will introduce the operators $-\Delta + 1$ and $M$, both subject to mixed boundary conditions, in detail in Section~\ref{sec:Delta and M}. The crucial change in our approach is that we define the smoothing operators within the functional calculus of $M$ instead of $-\Delta + 1$. Since $M$ acts on vector-fields by definition, the same is true for the smoothing operators that we define. Hence, we do not need to employ component-wise constructions. More concretely, our smoothing operators are simply resolvents of $M$, see Definition~\ref{def:smooting operators}. For them, the following beautiful identity, which is a consequence of Proposition~\ref{prop:M and Delta}, holds
    \begin{align}
    \label{eq:intro_intertwining}
        (1 + t^2 M)^{-1} T u = T (1 + t^2 (-\Delta + 1))^{-1} u, \qquad u\in \IW^{1,2}_\Ar{D}(O), t > 0.
    \end{align}
    With~\eqref{eq:intro_intertwining}, we can still commute the \enquote{inhomogeneous gradient} $T$ with the smoothing operator, similar to the commutation property using the Fourier transform on Eucliean space, but this operation will transform the \enquote{vector-valued} operator $M$ into the \enquote{scalar-valued} operator $-\Delta+1$ and vice versa.

    Also, we mention here that the operator $M$ has its place in the $T(b)$-type argument that is key to the solution of the Kato square root problem. We are going to perturb the operator $M$ by the elliptic coefficient matrix $$B = \begin{bmatrix}
        d & c \\
        b & A
    \end{bmatrix}
    \colon O \to \C^{ (n+1)m \times (n+1)m }$$ to obtain the (still sectorial) operator $MB$. Its resolvents will be used to define the $T(b)$-type test functions, see~\eqref{eq:def_b}.

    We tried to keep the present article as self-contained as possible. However, to avoid too much repetition, results and arguments that can literally be taken from~\cite{BEH-Kato} or~\cite{ISEM} will not be carried out again. Hence, the reader is advised to keep copies of these manuscripts handy.

    \subsection*{Notation}

    Give vectors $\xi, \eta \in \C^k$, put $\xi \cdot \eta \coloneqq \sum_{i=1}^k \xi_i \eta_i$. Their inner product can hence be written as $\langle \xi, \eta \rangle = \xi \cdot \overline{\eta}$.
    For $x\in \R^n$ and $r > 0$, write $\B(x,r)$ for the Euclidean ball centered in $x$ of radius $r$. If $B=\B(x,r)$, put $\rad(B) \coloneqq r$. Moreover, if $c > 0$ and $B$ is an Euclidean ball, write $cB$ for the concentric ball of radius $c\rad(B)$. Similarly, for $x\in \R^n$ and $\ell > 0$, write $\Q(x,\ell)$ for the Euclidean (axis-parallel) cube centered in $x$ of sidelength $\ell$. If $Q=\Q(x,\ell)$, put $\ell(Q) \coloneqq \ell$. Moreover, if $c > 0$ and $Q$ is an Euclidean cube, write $cQ$ for the concentric cube of sidelength $c\ell(Q)$.
    For a fixed (open) ambient set $O \subseteq \R^n$, write $\B_O(x,r) \coloneqq \B(x,r) \cap O$ and $\Q_O(x,r) \coloneqq \Q(x,r) \cap O$.
    For $\varphi \in (0, \pi)$, define the open sector of opening angle $2\varphi$ around the positive real axis by $\Sec_\varphi \coloneqq \{ z \in \C \setminus \{ 0 \} \colon |\arg(z)| < \varphi \}$. In the case $\varphi = 0$, put $\Sec_0 \coloneqq (0,\infty)$. The corresponding closed sectors are denoted by $\overline{\Sec}_\varphi$.
    Given $A \subseteq \R^n$ and $\delta > 0$, write $A_\delta \coloneqq \{ x \in \R^n \colon \dist(x,A) < \delta \}$ for the tube of size $\delta$ around $A$.
    Write $\dist(A,B)$ for the Euclidean distance of sets $A,B \subseteq \R^n$.
    For $A \subseteq \R^n$ measurable with $|A| > 0$ and $f$ an integrable function on $A$, write $(f)_A \coloneqq \fint_A f(y) \d y$ for the average of $f$ over $A$. Write $\Cont^\infty_0(\R^n)$ for the space of smooth and compactly supported functions on $\R^n$.
    Let $\mathcal{L}(H,K)$ denote the bounded linear operators between two Hilbert spaces $H$ and $K$, and put $\mathcal{L}(H) \coloneqq \mathcal{L}(H,H)$.
    The space $\H^\infty(\Xi)$ denotes the bounded, holomorphic functions on an open set $\Xi \subseteq \C$.

    \subsection*{Acknowledgments}

    The current manuscript is the outgrowth of a student project in the scope of the \enquote{27th international internet seminar} supervised by the first-named author and carried out by the remaining authors. We thank the organizers of this event, Moritz Egert, Robert Haller, Sylvie Monniaux and Patrick Tolksdorf, for providing the perfect surrounding for this project.

	\section{Geometric setup}
	\label{sec:geometry}

    We adopt the geometric setting used in~\cite{BEH-Kato}. For convenience of the reader, we summarize below the concepts that are necessary to understand our article.

    By $\cH^s(E)$, $s \in (0,n]$, we denote the $s$-dimensional \emph{Hausdorff measure} of $E\subseteq \R^n$ defined as follows. For $\varepsilon>0$, we put $$\cH^s_\varepsilon(E)\coloneqq \inf \Bigl\{ \sum_i \rad(B_i)^s: \bigcup_i B_i \supseteq E, \rad(B_i) \leq \varepsilon \Bigr\},$$ and since this value is increasing as $\varepsilon \to 0$, we define $\cH^s(E) \coloneqq \lim_{\varepsilon\to 0} \cH^s_\varepsilon(E)$, see \cite[Sec.~5.1]{Adams-Hedberg}.

    \begin{definition}
        \label{Def: AD}
        A closed set $D \subseteq \R^n$ is called \emph{Ahlfors--David regular} if there is comparability
        \begin{align}
        \label{eq: l-set}
        \cH^{n-1}(B \cap D) \approx \rad(B)^{n-1}
        \end{align}
        uniformly for all open balls $B$ of radius $\rad(B) < \diam(D)$ centered in $D$.
    \end{definition}

    \begin{remark}
        \label{Rem: AD}
        A practical interpretation is that Ahlfors--David regular sets behave $(n-1)$-dimensional at all scales. When $D$ is bounded, then, for any $C \in (0,\infty)$, the restriction on the radius may equivalently be changed to $\rad(B) \leq C$, at the cost of changing the implicit constants, see~\cite[Lem.~A.4]{BE}. Sets with the latter property for $C=1$ are called $(n-1)$-sets in~\cite{JW}.
    \end{remark}

    \begin{definition}
        \label{Def: ITC}
        An open set $\bO \subseteq \R^n$ is called \emph{interior thick} if
        \begin{align}
        \tag{ITC}
        \label{eq:ITC}
        |B \cap \bO| \gtrsim |B|
        \end{align}
        uniformly for all open balls $B$ of radius $\rad(B) \leq 1$ centered in $\bO$.
    \end{definition}

    In analogy with Remark~\ref{Rem: AD}, the restriction $\rad(B) \leq 1$ is arbitrary and can be changed to $\rad(B) \leq C$ for any fixed constant $C \in (0,\infty)$, up to a change in the implicit constant.

    Throughout the manuscript, we are going to employ the following notational convention.

    \begin{convention}[Boldface symbols]
        \label{conv:bold}
        If the underlying set $O$ of an elliptic system is supposed to be interior thick, we write $\bO$ instead of $O$. All other objects related to the interior thick set $\bO$, such as an elliptic system $L$, the Dirichlet boundary parts $\Ar{D}$ and so on, will likewise be denoted by the corresponding bold letters $\bL$, $\bAr{D}$ and so on.
    \end{convention}

    \subsection{Locally uniform domains}

    \begin{definition}
    \label{def: locally eps-delta}
    Let $\eps \in (0,1]$ and $\delta \in (0,\infty]$.  Let $O \subseteq \R^n$ be open and $N \subseteq \bd O$. Recall the notation $N_\delta  = \{z \in \R^n \colon \dist(z,N) < \delta \}$. Then $O$ is called \emph{locally an $(\eps,\delta)$-domain near $N$} if the following properties hold.

    \begin{enumerate}
    	\item All points  $x,y \in O \cap N_\delta$ with $|x-y| < \delta$ can be joined in $O$ by an \emph{$\eps$-cigar with respect to $\bd O \cap N_\delta$}, that is to say, a rectifiable curve $\gamma \subseteq O$ of length $\ell(\gamma) \leq \eps^{-1}|x-y|$ such that
    	\begin{align}
    	\label{eq: eps-delta}
    	\dist(z, \bd O \cap N_\delta) \geq \frac{\eps|z-x|\, |z-y|}{|x-y|} \qquad \mathrlap{(z \in \gamma).}
    	\end{align}

    	\item $O$ has \emph{positive radius near $N$}, that is, there exists $c > 0$ such that all connected components $O'$ of $O$ with $\bd O' \cap N \neq \emptyset$ satisfy $\diam(O') \geq c$.
    \end{enumerate}

    If the values of $\eps$ and $\delta$ need not be specified, then $O$ is simply called \emph{locally uniform near~$N$}.
    \end{definition}

    \begin{remark}
    \label{Rem: locally eps-delta}
    Definition~\ref{def: locally eps-delta} describes a quantitative local connectivity property of $O$ near $N$. For an illustration of $\eps$-cigars with respect to $\bd O$, the reader can refer, for instance, to~\cite[Fig.~3.1]{Triebel-Wavelets}. Having positive radius is of course only a restriction if $O$ has infinitely many connected components. %
    \end{remark}

    Locally uniform domains near $N$ are compared to other popular geometric frameworks for mixed boundary conditions in~\cite[Prop.~2.5]{BEH-Kato} and the subsequent remarks. We refer the reader to this discussion for more information.

    Locally uniform domains satisfy the following corkscrew condition around the Neumann part $N$.

    \begin{proposition}[Corkscrew condition]
        \label{prop:corkscrew}
        Suppose that $O \subseteq \R^n$ is open and locally an $(\eps,\delta)$-domain near $N \subseteq \partial O$. Then, there exists a constant $\kappa \in (0,1]$ such that:
        \begin{align}
            \forall x\in \overline{N_{\nicefrac{\delta}{2}}\cap O}, r \leq 1 \quad \exists z\in \B(x,r)\colon \quad \B(z,\kappa r) \subseteq O \cap \B(x,r).
        \end{align}
    \end{proposition}

    In Sections~\ref{sec:reduction}--\ref{sec:carleson}, we are going to establish Theorem~\ref{thm:main} first under the additional assumption that $O$ is interior thick.
    Afterwards, we are going to use the corkscrew condition to verify interior thickness of some auxiliary set in Section~\ref{sec:no_thickness}, which will enable us to get rid of the interior thickness condition in Theorem~\ref{thm:main}. Moreover, in combination with Ahlfors--David regularity of $D \coloneqq \partial O \setminus N$, the corkscrew condition yields \emph{porosity} (see the proposition for a definition) of the full boundary $\partial O$, see~\cite[Cor.~2.11]{BEH-Kato} for a proof.

    \begin{proposition}[Porosity of $\partial O$]
        \label{prop:porosity}
        Let $O \subseteq \R^n$ be open and $D \subseteq \partial O$ be closed. Suppose that $D$ is Ahlfors--David regular and that $O$ is locally uniform near $\partial O \setminus D$. Then $\partial O$ is \emph{porous}, that is to say, there exists $\kappa \in (0,1]$ with the property:
        \begin{align}
            \forall x\in \partial O, r\leq 1 \quad \exists z\in \B(x,r)\colon \quad \B(z, \kappa r) \subseteq \B(x,r) \setminus \partial O.
        \end{align}
    \end{proposition}

    In Section~\ref{sec:dyadic}, we are going to introduce so-called \enquote{dyadic structures}, which mimic the concept of dyadic cubes within an interior thick set $\bO$.
    Porosity of $\partial \bO$ will provide additional properties of the dyadic structures.

    \subsection{Sobolev spaces}

    The Hilbert space $\W^{1,2}(O)$ on an open set $O \subseteq \R^n$ is the collection of all $u \in \L^2(O)$ such that $\nabla u \in \L^2(O)^n$ with norm
    \begin{align}
    \label{intrinsic Sobolev norm}
     \|u\|_{\W^{1,2}(O)} \coloneqq \Big(\|u\|_{\L^2(O)}^2 + \|\nabla u\|_{\L^2(O)^n}^2 \Big)^{1/2}.
    \end{align}
    We introduce the subspace of functions that vanish on a subset of $\partial O$ as follows.

    \begin{definition}[Sobolev space subject to mixed BC]
    \label{Def: W12D}
    Let $O \subseteq \R^n$ be open and $D \subseteq \partial O$ be closed. The Hilbert space $\W^{1,2}_D(O)$ is the  $\W^{1,2}(O)$-closure of the set of test functions
    \begin{align}
    \label{Def: CDinfty}
     \Cont^\infty_D(O)\coloneqq \Bigl\{ u|_O: u\in \Cont^\infty_0(\R^n)\; \text{and} \; \dist(\supp(u),D)>0 \Bigr\}.
    \end{align}
    If $\Ar{D} = (D_i)_{i=1}^m$ is an array consisting of sets $D_i \subseteq \partial O$, then define the vector-valued space
    \begin{align}
        \IW^{1,2}_\Ar{D}(O) \coloneqq \bigoplus_{i=1}^m \W^{1,2}_{D_i}(O).
    \end{align}
    \end{definition}

    Due to the product structure, it is often possible to restrict one's attention to the case $m=1$, at least when pure function space questions are concerned. We will do so whenever possible.
    The given definition is standard in the modelling of mixed boundary conditions, see for instance~\cite[Sec.~4.1]{Ouhabaz}.
    In the geometric context of Theorem~\ref{thm:main}, the space $\W^{1,2}_D(O)$ can equivalently be defined by virtue of a more natural class of test functions, see~\cite[Sec.~7]{BHT}.

    For pure Dirichlet boundary conditions, we recover $\W^{1,2}_{\bd O}(O) = \W^{1,2}_0(O)$. To the contrary, in the case of pure Neumann boundary conditions it is possible that $\W^{1,2}_{\emptyset}(O)$ does \emph{not} coincide with $\W^{1,2}(O)$. Indeed, a disk with a slit serves as a counterexample~\cite[Ex.~1.1.10]{Diss}. A sufficient condition for the validity of the identity $\W^{1,2}(O) = \W^{1,2}_{\emptyset}(O)$ is the existence of a bounded Sobolev extension operator. This is, for instance, the case if $O$ is locally uniform, by virtue of Jones' result~\cite{Jones_ExtensionOperator}.

    The following result provides a bounded extension operator for the space $\W^{1,2}_D(O)$. It can be obtained by combining~\cite[Prop.~2.5]{BEH-Kato} with the main result of~\cite{BHT}. In the case $D = \emptyset$, it is compatible with Jones' result.

    \begin{theorem}[Extension operator]
        \label{Thm: W12D extension}
        Let $O \subseteq \R^n$ be open and $D \subseteq \bd O$ be closed. If $O$ is locally uniform near $\bd O \setminus D$, then there exists a bounded linear extension operator $\cE: \W^{1,2}_D(O) \to \W^{1,2}_D(\R^n)$. Moreover, the operator $\cE$ is local and homogeneous in the following sense: there are $c \geq 1$ and $r_0 > 0$, depending on geometry, such that, for $x\in \overline{O}$ and $r \in (0,r_0]$, there holds the estimate
        \begin{align}
            \| \nabla \cE u \|_{\L^2(\B(x,r))} \lesssim \| \nabla u \|_{\L^2(B_O(x, cr))}, \qquad u \in \W^{1,2}_D(O).
        \end{align}
    \end{theorem}

    Using the local and homogeneous estimates for the extension operator, a local Poincaré inequality can be deduced for $\W^{1,2}_D(O)$. This has been carried out in~\cite[Prop.~3.9 \& Rem.~3.10]{BCE}.

    \begin{corollary}[Local Poincaré inequality]
        \label{cor:local_poincare}
        In the setting of Theorem~\ref{Thm: W12D extension}, there are $c \geq 1$ and $r_0 > 0$ depending on geometry such that, for $x\in \overline{O}$ and $r \in (0,r_0]$, there holds the estimate
        \begin{align}
            \| u - (u)_{\B_O(x,r)} \|_{\L^2(\B_O(x,r))} \lesssim r \| \nabla u \|_{\L^2(\B_O(x,c r))}.
        \end{align}
    \end{corollary}

    \begin{remark}
        \label{rem:local_poincare}
        Using a telescoping sum, the following extension of Corollary~\ref{cor:local_poincare} is possible. Let $x \in \overline{O}$ and $\ell$ a positive integer such that $2^\ell r \leq r_0$. Then, there holds the estimate
        \begin{align}
            \| u - (u)_{\B_O(x,r)} \|_{\L^2(\B_O(x,2^\ell r))} \lesssim \bigl(\ell 2^{\nicefrac{n \ell}{2}} \bigr) 2^\ell r \| \nabla u \|_{\L^2(\B_O(x,c 2^\ell r))}.
        \end{align}
    \end{remark}

    \section{Elliptic operator}
	\label{sec:elliptic}

    In this section, we properly define the elliptic system appearing in the formulation of Theorem~\ref{thm:main}. Then, we collect some basic, but important, properties for it.
    Moreover, we define and discuss a counterpart to the Laplacian acting on vector fields, which will play a crucial role in the later analysis.

    \subsection{Elliptic systems in divergence form}
        \label{sec: operator L}

    The main player in the Kato square root problem is the \emph{elliptic system} of second order in divergence form
    $$ Lu = -\sum_{i,j=1}^n \partial_i (a_{ij} \partial_j u) - \sum_{i=1}^n \partial_i (b_i u) + \sum_{j=1}^n c_i \partial_i u + d u $$
    on $ O $ subject to Dirichlet boundary conditions on $ \Ar{D} \subseteq ( \partial O )^m $ and Neumann boundary conditions on $ ( \partial O )^m \setminus \Ar{D} $, where $ m \geq 1 $ denotes the size of the system.

    Throughout this article, we work under the following implicit assumption on the coefficients $ A $, $ b $, $ c $ and $ d $.

    \begin{assumption}{(B)}
        \label{ass: coeffs of L}
        The coefficients of $ L $ are bounded functions on $ O $, that is, $ A \in \L^\infty (O; \C^{ nm \times nm } ) $, $ b \in \L^\infty (O; \C^{nm \times m} ) $,  $ c \in \L^\infty (O; \C^{n \times nm} ) $ and $ d \in \L^\infty (O; \C^{m \times m}) $, and they satisfy the (inhomogeneous) G{\aa}rding inequality: there is some $ \lambda > 0 $ such that
        \begin{equation}
            \label{eq: def garding ineq}
            \mathfrak{a} (u,u)
            \geq \lambda ( \norm{u}_2^2 + \norm{ \nabla u }_2^2 ) \, ,
            \quad u \in \IW^{1,2}_\Ar{D} (O) \, ,
        \end{equation}
        where the form $\mathfrak{a}$ is given by
        \begin{align}
            \label{eq: def of sesqform a}
            \mathfrak{a} (u,v)
            = \int_O A \nabla u \cdot \nabla \bar{v}
            + bu \cdot \nabla \bar{v}
            + (c \nabla u) \bar{v}
            + (d u) \bar{v}
             \, ,
            \quad u,v \in \IW^{1,2}_\Ar{D} (O) \, .
        \end{align}
    \end{assumption}

    Formally, $ L $ can be written as the composition
    \begin{equation}
        \label{eq: L as formal compos}
        L u =
        \begin{bmatrix}
            1 & -\Div
        \end{bmatrix}
        \begin{bmatrix}
            d & c \\
            b & A
        \end{bmatrix}
        \begin{bmatrix}
            1 \\
            \nabla
        \end{bmatrix}
        u
        = T^* B T u \, ,
    \end{equation}
    where the inhomogeneous gradient is given by $ T = [1 \ \ \nabla ]^t $, has (formal) adjoint $ T^* = [1 \ \ -\Div ] $, and $ B $ denotes the matrix
    \begin{equation}
        \label{eq: def of B}
        B =
         \begin{bmatrix}
            d & c \\
            b & A
        \end{bmatrix}
        : O \to \C^{ (n+1)m \times (n+1)m } \,
    \end{equation}
    considered as a bounded multiplication operator on $ \L^2 (O)^{ (n+1)m } $.

    To make this precise, we define the gradient with appropriate boundary conditions as a closed operator on $ \L^2 (O)^m $, and set the divergence as its Hilbert space adjoint.
    	\begin{definition}[Gradient and divergence]
		\label{def:nabla and div}
        Denote by $ \nabla_\Ar{D} : \IW^{1,2}_\Ar{D} ( O ) \subseteq \L^2 ( O )^m \to~\L^2 ( O )^{nm} $ the gradient with Dirichlet boundary conditions on $ \Ar{D} $.
        This is a closed and densely
        defined operator, and we denote its (unbounded) adjoint by $ -\Div_\Ar{D} = (\nabla_\Ar{D})^* : \dom ( \Div_\Ar{D} ) \subseteq \L^2 ( O )^{nm} \to \L^2 ( O )^m $.

    \end{definition}

    Before we define the inhomogeneous gradient, we fix the convenient notation
    \begin{align}
        \label{def: N = (n+1)m }
        N = (n+1)m \, .
    \end{align}
    Elements $ F \in \L^2 ( O )^N $ will sometimes be referred to as vector fields.

    \begin{definition}[The operators $ T_\Ar{D} $ and $ T^*_\Ar{D} $]
        \label{def: operator T and T*}
        Define the inhomogeneous gradient $ T_\Ar{D} = ( 1 , \nabla_\Ar{D} )^t : \IW^{1,2}_\Ar{D} ( O ) \subseteq \L^2 ( O )^m \to \L^2 ( O )^N $ with Dirichlet boundary conditions on $ \Ar{D} $, and denote by $ T^*_\Ar{D} : \dom ( T^*_\Ar{D} ) \subseteq \L^2 ( O )^N \to \L^2 ( O )^m $ its adjoint.
        By a slight abuse of notation, we occasionally let $ T^*_\Ar{D} $ also denote the bounded extension $ T^*_\Ar{D} : \L^2 ( O )^N \to (\W^{1,2}_\Ar{D} ( O ))^* $.
    \end{definition}
    Note that $\| T_\Ar{D} u \|_2 = \| u \|_{\IW^{1,2}_\Ar{D}(O)}$ for $u \in \IW^{1,2}_\Ar{D}(O)$.
    If no confusion can arise, we sometimes write $T$ instead of $T_\Ar{D}$, and $T^*$ instead of $T^*_\Ar{D}$ to ease notation.

    \begin{definition}
        \label{def: def of L}
        We define the elliptic operator $ L $ as the composition
        \begin{equation}
            \label{eq: def of L}
            L = T_\Ar{D}^* B T_\Ar{D}
            : \dom (L) \subseteq \L^2 (O)^m
            \to \L^2 (O)^m
        \end{equation}
        with maximal domain in $ \L^2 (O)^m$.%
    \end{definition}
    While the direct definition is useful for algebraic manipulation, $ L $ can equivalently be defined as the operator on $ \L^2 (O)^m $ associated to the sesquilinear form $\mathfrak{a}$ from~\eqref{eq: def of sesqform a}. To this end, observe the identity
    \begin{align}
        \mathfrak{a}(u,v) = \int_O B T_\Ar{D} u \cdot \overline{ T_\Ar{D} v } \, ,
            \quad u,v \in \IW^{1,2}_\Ar{D} (O) \, .
    \end{align}
    Assumption~\ref{ass: coeffs of L} ensures that the form $ \mathfrak{a} $ is sectorial, so that the associated operator on $ \L^2 (O)^m $ is well-defined.
    For a proof of the equivalence of these two definitions see~\cite[Lemma~12.1]{seb_book}.
    Also associated to $ \mathfrak{a} $ is a bounded operator
    $$ \mathcal{L} : \IW^{1,2}_\Ar{D} (O) \to (\IW^{1,2}_\Ar{D} (O))^* $$
    extending $ L $ to $ \IW^{1,2}_\Ar{D} (O) $.
    As for $ L $, the bounded extension $ \mathcal{L} $ is given as the composition
    $$ \mathcal{L} = T^*_\Ar{D} B T_\Ar{D} \, , $$
    where this time $ T^*_\Ar{D} $ is understood in the \enquote{very weak sense}, that is, as the bounded operator $ T^*_\Ar{D} : \L^2 ( O )^N \to (\IW^{1,2}_\Ar{D} ( O ))^* $.

    One easily verifies that the adjoint $ L^* $ is the operator associated to the adjoint form $ \mathfrak{a}^* $, which shares the same properties as $L$ for Assumption~\ref{ass: coeffs of L} is invariant under taking adjoints.
    We collect some more important properties of the operator $ L $, and refer to~\cite[Chapter~6]{seb_book} and~\cite{ISEM} for background and proofs.
    \begin{proposition}
        \label{prop: L m accr and sectorial}
        The operator $ L $ is $ m $-accretive and sectorial.
        In particular, the resolvent family $ \{ (1+ t^2 L )^{-1} \}_{ t>0 } $ is uniformly bounded in $ \mathcal{L} ( \L^2 ( O ) )^m $.
    \end{proposition}
    By the von Neumann inequality and McIntosh's theorem, we obtain quadratic estimates for $ L $.
    \begin{corollary}
        \label{cor: L has quad est}
        The operator $ L $ has a bounded $ \H^\infty $-calculus and satisfies quadratic estimates.
    \end{corollary}
    Besides uniform boundedness of resolvents, we sometimes need to control terms of the form $ T_\Ar{D} ( 1+ t^2 L )^{-1} $, which can be achieved via the following simple result.
    \begin{lemma}
        \label{lem:adjvoint of nabla res}
        The family $ \{ tT_\Ar{D} ( 1+ t^2 L )^{-1} \}_{ t>0 } $ is uniformly bounded in $ \mathcal{L} ( \L^2 ( O )^m , \L^2 ( O )^N ) $, and its adjoint family is given by $ \{ t ( 1+ t^2 \mathcal{L} )^{-1} T_\Ar{D}^* \}_{ t>0 } $, where $ T_\Ar{D}^* $ is the bounded operator $ \L^2 ( O )^N \to (\IW^{1,2}_\Ar{D} ( O ))^* $.
    \end{lemma}

    \subsection{The operators \texorpdfstring{$ -\Delta_\Ar{D} + 1 $}{-Delta + 1} and \texorpdfstring{$ M_\Ar{D} $}{M}}
	\label{sec:Delta and M}

    Setting $ B \equiv \mathrm{id}_{ \C^N } $, we obtain the inhomogeneous Laplacian $ - \Delta_\Ar{D} +1 $ given by the composition $ -\Delta_\Ar{D} +1 = T^*_\Ar{D} T_\Ar{D} $ with
    corresponding form
    \begin{equation}
        \label{eq: form for laplace}
        \mathfrak{a} (u,v)
        = \scal{ T_\Ar{D} u , T_\Ar{D} v } \, ,
        \quad u,v \in \IW^{1,2}_\Ar{D} (O) = \dom ( T_\Ar{D} ) \, .
    \end{equation}
    Considering the sesquilinear form derived from $ T^*_\Ar{D} $,
    \begin{align}
        \label{eq: sesq form for M}
        \mathfrak{b} (F,G)
        = \scal{ T^*_\Ar{D} F , T^*_\Ar{D} G } \, ,
        \quad F,G \in \dom ( T^*_\Ar{D} ) \, ,
    \end{align}
    we obtain a natural operator $ M_\Ar{D} $ acting on vector fields $ F \in \L^2 ( O )^N $.

	\begin{definition}[The operator $M_\Ar{D}$]
		\label{def:M}
        Define the operator $ M_\Ar{D} : \dom ( M_\Ar{D} ) \subseteq \L^2 ( O )^N \to~\L^2 ( O )^N $ as the composition $ M_\Ar{D} = T_\Ar{D} T^*_\Ar{D} $ with maximal domain
        $$ \dom ( M_\Ar{D} ) = \bigl\{ F \in \dom ( T^*_\Ar{D} ) : T^*_\Ar{D} F \in \dom  ( T_\Ar{D} ) \bigr\} . $$
	\end{definition}

    \begin{remark}
        \label{rem: T* is surj}
        We claim that $ \ran(M_\Ar{D}) = \ran(T_\Ar{D}) $, and this is a closed subspace of $ \L^2 (O)^N $. First, closedness of $\ran(T_\Ar{D})$ is equivalent to the fact that the Sobolev space $\IW^{1,2}_\Ar{D}(O)$ is complete. Second, $\ran(T_\Ar{D}) \subseteq \ran(M_\Ar{D})$ follows from invertibility of $-\Delta_\Ar{D} + 1$. Indeed, let $u \in \IW^{1,2}_\Ar{D}(O)$ and write $u = (-\Delta_\Ar{D} + 1) v = T_\Ar{D}^* T_\Ar{D} v$ for some $v \in \dom(-\Delta_\Ar{D} + 1)$. Hence, $T_\Ar{D} u = T_\Ar{D} T_\Ar{D}^* T_\Ar{D} v = M_\Ar{D} F$, where $F \coloneqq T_\Ar{D} v \in \dom(M_\Ar{D})$.
    \end{remark}

    Formally, $ M_\Ar{D} $ is given by the operator matrix $
    \begin{bmatrix}
        1 & -\Div_\Ar{D} \\
        \nabla_\Ar{D} & -\nabla_\Ar{D} \Div_\Ar{D} \,
    \end{bmatrix}.$
    Note, however, that for $ F = ( F_0 , F' )^t \in \L^2 ( O )^m \times \L^2 ( O )^{nm} $, $ F \in \dom ( M_\Ar{D} ) $ does not imply $ F' \in \dom ( \nabla_\Ar{D} \Div_\Ar{D} ) $, but only $ F_0 - \Div_\Ar{D} F' \in \dom ( \nabla_\Ar{D} ) $.

    From the sesquilinear form $ \mathfrak{b} $, the operator $ M_\Ar{D} $ inherits the same desirable properties as the inhomogeneous Laplacian.

    \begin{proposition}
        \label{prop: M is self-adjoint}
        The operator $ M_\Ar{D} $ is self-adjoint and non-negative.
        In particular, $ M_\Ar{D} $ has a bounded $ \H^\infty $-calculus on its (closed) range and satisfies quadratic estimates as well as the square root property
        \begin{align}
            \label{eq: square root ppty for M}
            \normm{ \sqrt{ M_\Ar{D} } F }_2
            = \normm{ T_\Ar{D}^* F }_2 \, ,
            \quad F \in \dom ( T_\Ar{D}^* )
            = \dom ( \sqrt{ M_\Ar{D} } ) \, .
        \end{align}
    \end{proposition}

    The following intertwining property between $ M_\Ar{D} $ and $ -\Delta_\Ar{D} + 1 $ is essential for our arguments in Section~\ref{sec:principal part approximation}.

	\begin{proposition}[Intertwining between $M_\Ar{D}$ and $-\Delta_\Ar{D} + 1$]
		\label{prop:M and Delta}
		Let $f\in \H^\infty(\Sec_\varphi)$ for some $\varphi \in (0,\pi)$.
		Then, one has the intertwining relation $$T_\Ar{D} f(-\Delta_\Ar{D} + 1) u = f(M_\Ar{D}) T_\Ar{D} u \qquad\mathrlap{(u\in \IW^{1,2}_\Ar{D}(O)).}$$
	\end{proposition}
    \begin{proof}
        We suppose familiarity with the construction and basic properties of the sectorial functional calculus. A reader who is not yet familiar with this concept can consult the monograph~\cite{Haase} or~\cite[Sec.~5]{ISEM}.

        Fix $u\in \IW^{1,2}_\Ar{D}(O)$. Also, let $\varphi \in (0,\pi)$ and $f\in \H^\infty(\Sec_\varphi)$. %
        We comment first why the identity from the statement is well-defined. First of all, $f(-\Delta_\Ar{D} + 1)$ is well-defined since the operator $-\Delta_\Ar{D} + 1$ is sectorial of angle zero and injective (even invertible). Next, write $u = (-\Delta_\Ar{D} + 1)^{-\frac{1}{2}} v$ for a suitable $v \in \L^2(O)^m$, which uses the Kato square root property for $-\Delta_\Ar{D} + 1$. Then, the representation $f(-\Delta_\Ar{D} + 1) u = (-\Delta_\Ar{D} + 1)^{-\frac{1}{2}} f(-\Delta_\Ar{D} + 1) v$ shows that $f(-\Delta_\Ar{D} + 1)u$ is indeed an element of $\IW^{1,2}_\Ar{D}(O) = \dom(T_\Ar{D})$, where we used a basic commutation property of the functional calculus and boundedness of the $\H^\infty$-calculus for $-\Delta_\Ar{D} + 1$. Therefore, the left-hand side of the claimed identity is well-defined. For the right-hand side, note that $M_\Ar{D}$ is in general not injective. However, if we restrict $M_\Ar{D}$ to $\overline{\ran(M_\Ar{D})} = \ran(M_\Ar{D})$ (see Remark~\ref{rem: T* is surj} for this identity), then $M_\Ar{D}$ becomes injective and $f(M_\Ar{D})$ can be defined as an operator on $\ran(M_\Ar{D})$.
        We are not going to distinguish $M_\Ar{D}$ from its injective part notation-wise.
        Since $T_\Ar{D} u \in \ran(M_\Ar{D})$ by Remark~\ref{rem: T* is surj}, $f(M_\Ar{D})$ is indeed defined on the \enquote{gradient field} $T_\Ar{D} u$ appearing in the statement.

        We split the proof into 3 steps.

        \textbf{Step 1}: Show $T_\Ar{D} (z - (-\Delta_\Ar{D} + 1))^{-1} u = (z - M_\Ar{D})^{-1} T_\Ar{D} u$ for all $z \in \C \setminus [0, \infty) $.

        The operators $-\Delta_\Ar{D} + 1$ and $M_\Ar{D}$ are sectorial of angle zero, hence their resolvents for $z \in \mathbb{C} \setminus [0,\infty)$ are well-defined.
        Put $v \coloneqq (z - (-\Delta_\Ar{D} + 1))^{-1} u$ for convenience. We can apply $(z - (-\Delta_\Ar{D}+1))$ to the definition of $v$ to find $(z - (-\Delta_\Ar{D} + 1)) v = u$. Since $u\in \IW^{1,2}_\Ar{D}(O)$, we can apply $T_\Ar{D}$ to the last identity. In conjunction with $T_\Ar{D} (-\Delta_\Ar{D} + 1) = M_\Ar{D} T_\Ar{D}$, this gives
        \begin{align}
            T_\Ar{D} u = T_\Ar{D} (z - (-\Delta_\Ar{D} + 1)) v = (z - M_\Ar{D}) T_\Ar{D} v.
        \end{align}
        Finally, we apply $(z - M_\Ar{D})^{-1}$ to the last identity and plug in the definition of $v$ to conclude
        \begin{align}
            (z - M_\Ar{D})^{-1} T_\Ar{D} u = T_\Ar{D} v = T_\Ar{D} (z - (-\Delta_\Ar{D} + 1))^{-1} u.
        \end{align}

        \textbf{Step 2}: Show the intertwining relation for $f \in \H_0^\infty(\Sec_\varphi)$.

        Since $f \in \H_0^\infty(\Sec_\varphi)$, we can write $f(M_\Ar{D})$ and $f(-\Delta_\Ar{D} + 1)$ via Cauchy integrals. Then, the claim follows from Step~1 by standard arguments.

        \textbf{Step 3}: The general case $f\in \H^\infty(\Sec_\varphi)$.

        Recall
        the standard regularizer $e(z) = z(1 + z)^{-2}$ for the $\H^\infty$-calculus of an injective sectorial operator. By definition, $f(-\Delta_\Ar{D} + 1) u = e(-\Delta_\Ar{D} + 1)^{-1} (ef)(-\Delta_\Ar{D} + 1) u$. As discussed at the start of the proof, $f(-\Delta_\Ar{D} + 1) u$ is again in $\IW^{1,2}_\Ar{D}(O)$. Therefore, we can apply $T_\Ar{D}$ to the last identity to get $T_\Ar{D} f(-\Delta_\Ar{D} + 1) u = T_\Ar{D} e(-\Delta_\Ar{D} + 1)^{-1} (ef)(-\Delta_\Ar{D} + 1) u$. By construction of the functional calculus, $v \coloneqq (ef)(-\Delta_\Ar{D} + 1) u$ is in the range of $e(-\Delta_\Ar{D} + 1)$. Furthermore, $v\in \IW^{1,2}_\Ar{D}(O)$. Put $w \coloneqq e(-\Delta_\Ar{D} + 1)^{-1} v$. We claim that also $w\in \IW^{1,2}_\Ar{D}(O)$. Indeed, $w = f(-\Delta_\Ar{D} + 1) u \in \IW^{1,2}_\Ar{D}(O)$ as observed above.

        Next, we show $T_\Ar{D} e(-\Delta_\Ar{D} + 1)^{-1} v = e(M_\Ar{D})^{-1} T_\Ar{D} v$. To this end, write $T_\Ar{D} v = T_\Ar{D} e(-\Delta_\Ar{D} + 1) w = e(M_\Ar{D}) T_\Ar{D} w$, where we used Step~2 for the last manipulation. Therefore, $T_\Ar{D} v \in \ran(e(M_\Ar{D}))$ and we conclude $e(M_\Ar{D})^{-1} T_\Ar{D} v = T_\Ar{D} w = T_\Ar{D} e(-\Delta_\Ar{D} + 1)^{-1} v$ as desired.

        In summary, we deduce with another application of Step~2 to $ef \in \H^\infty_0(\Sec_\varphi)$ that
        \begin{align}
            T_\Ar{D} f(-\Delta_\Ar{D} + 1) u &= T_\Ar{D} e(-\Delta_\Ar{D} + 1)^{-1} (ef)(-\Delta_\Ar{D} + 1) u \\
            &= e(M_\Ar{D})^{-1} T_\Ar{D} (ef)(-\Delta_\Ar{D} + 1) u \\
            &= e(M_\Ar{D})^{-1} (ef)(M_\Ar{D}) T_\Ar{D} u \\
            &= f(M_\Ar{D}) T_\Ar{D} u. \qedhere
        \end{align}
    \end{proof}

	\begin{remark}
		The last proposition is our substitute for the observation made in~\cite[Rem.~4.5]{AKM-QuadraticEstimates}.
	\end{remark}

    One important property of elliptic operators in divergence form is their off-diagonal decay.
    A standard way to derive them is by using Davies' trick. In our geometric setting, this was done in~\cite{Egert_Lp_Kato,Bechtel_Lp_Kato} for the families $ \{ \e^{-t^2 L} \}_{ t>0 } $ and $ \{ t \nabla \e^{-t^2 L} \}_{ t>0 } $. The same proof also applies to the family $ \{ t\e^{-t^2 L} \}_{ t>0 } $. Indeed, it suffices to observe that the coercivity condition used in the treatment of the gradient family is in fact an inhomogeneous coercivity condition.
    Using the Laplace transform, these bounds can be transfered to the respective resolvent families.

	\begin{proposition}[Off-diagonal estimates]
		\label{prop:ode}
        The families $ \{ t(1+ t^2 L )^{-1} \}_{ t>0 } $ and $ \{ t \nabla (1+ t^2 L )^{-1} \}_{ t>0 } $ satisfy $ \L^2 $ off-diagonal estimates,
        that is, there exist constants $ c, C >0 $ such that, for all $ t>0 $, all measurable subsets $ E, F \subseteq O $ and all $ u \in \L^2 ( O ) $, there holds
        \begin{equation*}
            \normm{ \one_E t( 1+ t^2 L )^{-1} \one_F u }_2
            + \normm{ \one_E t \nabla
            ( 1+ t^2 L )^{-1} \one_F u }_2
            \leq C \e^{ -c \frac{ \dist (E,F) }{t} }
            \norm{ \one_F u }_2 \, .
        \end{equation*}
        By duality, also the adjoint family $ \{ t ( 1+ t^2 \mathcal{L} )^{-1} \Div \}_{ t>0 } $
        satisfies $ \L^2 $ off-diagonal estimates.
	\end{proposition}

	\subsection{Perturbations of \texorpdfstring{$M_{\Ar{D}}$}{M}}
	\label{sec:M perturb}

    We perturb the operator $ M_\Ar{D} $ by the elliptic coefficient function $B$. The resulting operator $ M_\Ar{D} B $ is again sectorial, a fact that we are going to exploit in the construction of $ T(b) $-type test functions in Section~\ref{sec:carleson}.

    For instance in~\cite[Prop.~6.1.17]{Diss}, it is shown that the perturbation $ DB $ of a self-adjoint operator $ D $ is bisectorial.
    We make a minor adjustment to their argument to prove that, if $ D $ is non-negative, then $ DB $ is even sectorial.
    We formulate the subsequent result only for the choice $D = M_\Ar{D}$.
    Let us mention that, if $B$ were even strongly elliptic and not merely G{\aa}rding elliptic, then the subsequent result could be proved in an easier manner~\cite[Prop.~6.9]{seb_book}.

	\begin{lemma}
		\label{lem:M perturbation}
		If the coefficients $ B $ satisfy the G{\aa}rding ellipticity condition on $ \ran T_\Ar{D} $ as stated in Assumption~\ref{ass: coeffs of L}, then the operator $ M_\Ar{D} B$ is sectorial.
        Moreover, the family $ \{ t T^*_\Ar{D} B (1+ t^2 M_\Ar{D} B )^{-1} \}_{ t>0 } $ is uniformly bounded in $ \mathcal{L} ( \L^2 ( O )^N , \L^2 ( O )^m ) $.
	\end{lemma}
    \begin{proof}
        To ease notation, write $ M $, $T$, and $T^*$ for the operators $ M_\Ar{D} $, $T_\Ar{D}$, and $T^*_\Ar{D}$, and write $\| \cdot \|$ instead of $\| \cdot \|_2$. Recall that $ B $ satisfies Gårding's inequality on $ \ran(T) = \cl{ \ran(M) } $.
        The same is true for $B^*$ because Assumption~\ref{ass: coeffs of L} is invariant under taking the adjoint.
        From~\cite[Prop.~6.1.17]{Diss}, we have the (non-orthogonal) topological decompositions of $ \L^2 (O)^N $ given by
        \begin{align}
            \label{eq: top splitting 1}
            \L^2 (O)^N
            &= \ker (MB) \oplus \overline{ \ran (MB) } \, , \\
            \label{eq: top splitting 2}
            \L^2 (O)^N
            &= \ker ( B^* M) \oplus \overline{ \ran ( B^* M) } \, .
        \end{align}
        Now, set
        \begin{equation}
            \label{eq: sectoriality angle for MB}
            \omega
            \coloneqq \sup_{ 0 \neq F \in \overline{ \ran M } }
            |\arg \scal{ BF , F }| \in [0, \tfrac{ \pi }{2} ) \,  ,
        \end{equation}
        and take $ \phi \in ( \omega , \pi ) $.
        Using the decompositions above, we prove that, for all $ z \in \C \setminus \overline{\Sec}_\phi $, we have the estimates
        \begin{align}
            \label{eq: estimate for sec of MB 1}
            \normm{ ( z - MB ) U }
            &\gtrsim \abs{ z } \norm{U} \, ,
            \quad U \in \dom (MB) \, , \\
            \label{eq: estimate for sec of MB 2}
            \normm{ ( \bar{ z } - B^* M ) U }
            &\gtrsim \abs{ z } \norm{U} \, ,
            \quad U \in \dom ( B^* M ) \, .
        \end{align}
        From this, we deduce that $ MB $ is sectorial as follows.
        Owing to~\eqref{eq: estimate for sec of MB 2}, we find that $\overline{z} - B^* M$ is injective and that its range is closed.
        Since $ B $ is bounded, we have $ ( B^* M )^* = MB $.
        Hence,~\eqref{eq: estimate for sec of MB 1} yields that $\overline{z} - B^* M$ has dense range. Summarizing the two facts, we find $ \bar{ z } \in \rho ( B^* M ) $.
        Taking the estimate~\eqref{eq: estimate for sec of MB 2} into account, this shows sectoriality of $B^* M$.
        Next, using that $ B^* M $ is closed, it follows
        \begin{equation}
            \label{eq: adjoint of MB}
            (MB)^* = ( B^* M )^{**}
            = \text{cl} ( B^* M ) = B^* M \, .
        \end{equation}
        Thus, arguing similarly as above entails sectoriality of $MB$.

        We turn our attention to proving~\eqref{eq: estimate for sec of MB 1}. Let $ z \in \C \setminus \overline{\Sec}_\phi $.
        Let $ U \in \dom (MB) $ and, according to~\eqref{eq: top splitting 1}, write $ U = V+W $ with $ V \in \ker (MB) $ and $ W \in \cl{ \ran (MB) } \cap \dom (MB) $.
        Note that $BW \in \dom(M)$.
        Hence, by positivity of $ M $, it holds $ \scal{ BW , MB W } \geq 0 $.
        Moreover, since $W \in \overline{\ran(MB)} \subseteq \overline{\ran(M)}$, we have $ \scal{ BW , W } \in \overline{\Sec}_\omega $. Therefore, $ \bar{ z } \scal{ BW , W } \in \C \setminus \overline{\Sec}_{ \phi - \omega } $. Cover $\C \setminus \overline{\Sec}_{ \phi - \omega }$ by the set $ \C \setminus (\overline{\Sec}_{ \phi - \omega } \cup -\overline{\Sec}_{ \phi - \omega })$, which is the complement of a bisector, and the (closed) sector around the negative real axis $ -\overline{\Sec}_{ \phi - \omega }$. In the first case, if $w \in \C \setminus (\overline{\Sec}_{ \phi - \omega } \cup -\overline{\Sec}_{ \phi - \omega })$, then there holds comparability
        \begin{align}
            \label{eq:pert_M_comp1}
            |w| \approx_{\phi - \omega} |\Im w|
        \end{align}
        as in~\cite[Prop.~6.1.17]{Diss}. Otherwise, if $w\in -\overline{\Sec}_{ \phi - \omega }$, there holds comparability
        \begin{align}
            \label{eq:pert_M_comp2}
            |w| \approx |\Im w| - \Re w.
        \end{align}
        Suppose for the moment $\langle BW, z W\rangle \in -\overline{\Sec}_{ \phi - \omega }$.
        Then, from $ W \in \cl{ \ran(M) } = \ran(T) $ in conjunction with Gårding's inequality,~\eqref{eq:pert_M_comp2}, positivity of $M$, and boundedness of $B$, we get
        \begin{align}
            \abss{ z } \norm{W}^2
            &\les \abss{ \bar{ z } }
            \Re \scal{ BW , W } \\
            &\les \abss{ \scal{ BW , z  W } } \\
            &\approx \abss{ \Im \scal{ BW , z  W } }
            + \Re \scal{ BW , - z W } \\
            &\leq \abss{ \Im \scal{ BW , ( z - M B) W } }
            + \Re \scal{ BW , ( M B - z ) W }\\
            &\leq 2 \norm{BW} \normm{ ( z - M B ) W } \\
            &\les \norm{W} \normm{ ( z - M B ) W } \, .
        \end{align}
        Otherwise, if $\langle BW, z W\rangle$ belongs to the complement of the bisector, a similar calculation using~\eqref{eq:pert_M_comp1} instead of~\eqref{eq:pert_M_comp2} yields
        \begin{align}
            \abss{ z } \norm{W}^2 \lesssim_{\phi-\omega} \abss{ \Im \scal{ BW , ( z - M B) W }} \lesssim \norm{W} \normm{ ( z - M B ) W } \, .
        \end{align}
        Combining both cases yields $ \norm{ ( z - MB) W } \gtrsim \abs{ z } \norm{W} $.
        Furthermore, $\norm{ ( z - MB) V } = \abs{ z } \norm{v}$ for $V \in \ker(MB)$.
        Since $ ( z - MB) W \in \cl { \ran (MB) } $ and $MB V = 0 \in \ker(MB)$, the topological decomposition~\eqref{eq: top splitting 1} gives
        \begin{align}
            \norm{ ( z - MB) U }
            \approx \norm{ ( z - MB) V }
            + \norm{ ( z - MB) W }
            \gtrsim \abs{ z } ( \norm{V} + \norm{W} )
            \approx \abs{ z } \norm{U} \, ,
        \end{align}
        which completes the proof of~\eqref{eq: estimate for sec of MB 1}.

        To prove~\eqref{eq: estimate for sec of MB 2}, we split $ U \in \dom ( B^* M ) $ as $ U = V + Z $ with $ V \in \ker ( B^* M ) $ and $ Z \in \cl{ \ran ( B^* M ) } \cap \dom ( B^* M ) $.
        Since $ \norm{ B^* F} \approx \norm{F} $ on $ \overline{\ran (M)} $, we have $\cl { \ran ( B^* M ) } = B^* ( \cl{ \ran (M) } )$. Therefore, $ Z = B^* W $ for some $ W \in \cl { \ran (M) } \cap \dom ( B^* M B^* ) $.
        Reusing the calculation leading to~\eqref{eq: estimate for sec of MB 1} from above with $B$ replaced by $B^*$, we have
        \begin{align}
            \abss{ \Bar{ z } } \norm{ B^* W }^2
            \les \abs{ z } \norm{W}^2
            &\lesssim \norm{ B^* W }
            \norm{ ( \bar{ z } - M B^* ) W } \\
            &\les \norm{ B^* W }
            \norm{ B^* ( \bar{ z } - M B^* ) W },
        \end{align}
        where we used $(\bar{ z } - M B^* ) W \in \overline{\ran(M)}$ in the last step.
        Thus, plugging in the definition of $Z$, deduce $ \abss{ \bar{ z } } \norm{Z} \les \normm{ ( \bar{ z } - B^* M ) Z } $. Then, as before, we deduce from the splitting~\eqref{eq: top splitting 2} the estimate~\eqref{eq: estimate for sec of MB 2}.

        Finally let us prove uniform boundedness of $ t T^* B (1+ t^2 M B )^{-1} $.
        Observe that $ B ( 1+ t^2 M B )^{-1} $ maps into the domain of $ M $, and that $M$ is associated to the sesquilinear form $\mathfrak{b}$.
        Thus, for $ F \in \L^2 ( O )^N $, we have
        \begin{align}
            \normm{ t T^* B (1+ t^2 M B )^{-1} F }^2
            &= \mathfrak{b} ( t^2 B (1+ t^2 M B )^{-1} F ,
            B (1+ t^2 M B )^{-1} F ) \\
            &= \scal{ t^2 M B (1+ t^2 M B )^{-1} F ,
            B (1+ t^2 M B )^{-1} F } \\
            &\lesssim \norm{ (1 - (1+ t^2 M B )^{-1} ) F }
            \norm{ (1+ t^2 M B )^{-1} F }
            \les \norm{F}^2.
        \end{align}
        The last two steps make use of the Cauchy--Schwarz inequality, the boundedness of $ B $ and uniform boundedness of the resolvents $ (1+ t^2 M B )^{-1} $ in $ \mathcal{L} ( \L^2 ( O )^N ) $, where the latter follows by sectoriality of $M B$ established in the first part of the proof.
    \end{proof}

    \section{Review on dyadic structures}
	\label{sec:dyadic}

    Dyadic cubes are the foundation of dyadic harmonic analysis in $\R^n$.
    In this section, we recall a substitute on interior thick open sets for them, which will be fundamental for the analysis presented in Sections~\ref{sec:principal part approximation} and~\ref{sec:carleson}.

    To put things into a broader context, we note that an open and interior thick set $\bO$ equipped with the restricted Euclidean metric and restricted Lebesgue measure is a locally doubling space, that is, there exists $C \geq 2$ such that for all $x\in \bO$ and $r \in (0,1]$ it holds
        \begin{align}
        \label{eq:doubling}
            \abs{\B_\bO(x,2r)} \leq C \abs{\B_\bO(x,r)}.
        \end{align}
    In~\cite[Thm 11]{ChristDyadic}, Christ has constructed a full dyadic structure in spaces that fulfill the doubling property~\eqref{eq:doubling}, even for all $r\in (0,\infty)$.
    With the local doubling property from above, the same construction gives a dyadic structure on small scales, that is, we can decompose the space $\bO$ into small dyadic \enquote{cubes} but not into large dyadic \enquote{cubes}.
    For our formulation here, we combine~\cite[Thm.~7.4]{Laplace-Extrapolation} with the boundary estimate~\cite[Lem.~5.2]{BEH-Kato} and the rescaling argument from~\cite[Prop. 2.12]{DeltaArg}. The boundary estimate, that is, property~(5) in the subsequent proposition, is a consequence of porosity of $\partial\bO$ established in Proposition~\ref{prop:porosity}. The remaining results hold for $\bO$ merely interior thick.

	\begin{proposition}[Existence of a dyadic structure]
		\label{prop:dyadic structure}
        Suppose that $\bO$ is interior thick and that $\partial \bO$ is porous. Then, there exists a collection $\{Q^k_\alpha\ :\  k \in \N_0,\ \alpha \in I_k\}$ of open
        sets, where $I_k$ are countable index sets, and constants $a_0$, $a_1$, $\eta$, $C > 0$ such that the following hold:
        \begin{enumerate}
            \item For each $k \in \N_0$, it holds $\abs{\bO \setminus \cup_{\alpha \in I_k} Q^k_{\alpha}} = 0$.
            \item If $\ell \geq k$, then, for each $\alpha \in I_k$ and each $\beta \in I_\ell$, either $Q^\ell_\beta \subseteq Q^k_\alpha$ or $Q^\ell_\beta \cap Q^k_\alpha = \emptyset$.
            \item If $\ell \leq k$, then, for each $\alpha \in I_k$, there is a unique $\beta \in I_\ell$ such that $Q^k_\alpha \subseteq Q^\ell_\beta$.
            \item For each $Q^k_\alpha$, $k \in \N_0$, $\alpha \in I_k$, there exists $z^k_\alpha \in \bO$ such that $$\Q(z^k_\alpha, a_0 2^{-k}) \cap \bO \subseteq Q^k_\alpha \subseteq \Q(z^k_\alpha, a_1 2^{-k}) \cap \bO.$$
            \item If $k \in \N_0$, $\alpha \in I_k$, and $s>0$, then
                \begin{align}
                    \abs{\{x \in Q^k_\alpha\ :\ \mathrm{d}(x, \R^n \setminus Q^k_\alpha) \leq s2^{-k}\}} \leq Cs^{\eta}\abs{Q^k_\alpha}.
                \end{align}
        \end{enumerate}
	\end{proposition}

    \begin{definition}[Dyadic \enquote{cubes}]
    \label{eq:dyadic_cubes}
        In the setting of Proposition~\ref{prop:dyadic structure}, put $\Box_{2^{-k}} \coloneqq \{Q^k_\alpha\ :\ \alpha \in I_k \}$. Members of this collection are called \emph{dyadic \enquote{cubes}} of \emph{generation $2^{-k}$}. For $t \in (2^{-(k+1)},2^{-k}]$, define also $\Box_{t} \coloneqq \Box_{2^{-k}}$. In this case, the \emph{sidelength} of $Q \in \Box_{t}$ is $\ell(Q) \coloneqq 2^{-k}$. Finally, the collection $\Box$ of all (local) \emph{dyadic \enquote{cubes}} is given by $\bigcup_k \Box_{2^{-k}}$.
    \end{definition}

    \begin{definition}[Children and parents]
    \label{def:cube_family}
        Let $\ell \geq k$, $R \in \Box_{2^{-\ell}}$ and $Q \in \Box_{2^{-k}}$. If $R \subseteq Q$, then $R$ is called a \emph{child of $R$}, and $Q$ is called the \emph{(unique) parent cube} of $R$.
    \end{definition}

    With $R$ and $Q$ as in the preceding definition, Proposition~\ref{prop:dyadic structure}~(2) asserts that $R$ is either a child of $Q$ or disjoint to $Q$. Moreover, Proposition~\ref{prop:dyadic structure}~(3) asserts the existence of a parent cube to $R$. In conjunction with~(2), this parent cube has to be unique, which justifies our terminology of unique parent cubes in Definition~\ref{def:cube_family}.

    \begin{definition}[Double \enquote{cube}]
        If $Q \in \Box$, then
        define the \emph{doubled \enquote{cube}} $2Q$ via $2Q \coloneqq \{ x \in \bO \colon \dist(x, Q) < \ell(Q) \}$.
    \end{definition}

    Using property~(4) of dyadic \enquote{cubes}, interior thickness of $\bO$, and elementary geometric considerations, we find $|2Q| \lesssim |Q|$.

    \begin{remark}
        \label{rem:dyadic null}
        (1) and (2) in Proposition~\ref{prop:dyadic structure} imply that, outside of a null set $\mathfrak{N}$, each $x \in \bO \setminus \mathfrak{N}$ is uniquely contained in a \enquote{cube} $Q \in \Box_t$ for each $t\in (0,1]$.
    \end{remark}

    In the following definition, $\L^1_{\loc}(\bO)$ denotes the collection of measurable functions which are integrable over bounded subsets of $\bO$.

    \begin{definition}
        \label{def:dyadic averaging operator}
        For $t\in (0,1]$, $x \in \bO \setminus \mathfrak{N}$, and $u \in \L^1_{\mathrm{loc}}(\bO)$ define the \emph{dyadic averaging operator} by
            \begin{align}
                (\mathcal{A}_t u)(x) := \fint_{Q_{x,t}} u(y) \d y,
            \end{align}
        where $Q_{x,t} \in \Box_t$ is the unique \enquote{cube} containing $x$. For $F \in \L^1_{\mathrm{loc}}(\bO)^k$, via a minor abuse of notation, we define $\mathcal{A}_tF$ by component-wise action, that is
            \begin{align}
                (\mathcal{A}_t F)(x) := \bigl( (\mathcal{A}_t F_i)(x) \bigr)_{i=1}^k,
            \end{align}
        where $F = (F_i)_{i=1}^k$.
    \end{definition}
    We investigate some first properties.

    \begin{proposition}
        \label{prop:contraction}
        $\mathcal{A}_t^2 = \mathcal{A}_t$, and $\mathcal{A}_t$ is a contraction on $\L^2(\bO)$ for all $t \in (0,1]$.
    \end{proposition}
    \begin{proof}
    	Fix $t\in (0,1]$.
        For the projection property, note that
        \begin{equation*}
            (\mathcal{A}_t^2 u)(x) = \fint_{Q_{x,t}}\left(\fint_{Q_{z,t}} u(y) \d y\right) \d z = \fint_{Q_{x,t}} (u)_{Q_{x,t}}\ \d z = (u)_{Q_{x,t}} = \mathcal{A}_tu(x),
        \end{equation*}
        where we used $Q_{z,t} = Q_{x,t}$ for $z \in Q_{x,t}$ in the second step.

        For the second claim, decomposing $\bO \setminus \mathfrak{N}$ into the dyadic cubes $\Box_t$ and applying Jensen's inequality, we see for $u \in \L^2(\bO)$ that
        \begin{align}
            \norm{\mathcal{A}_tu}^2_2 = \sum_{Q \in \Box_t}\int_{Q}\abs{\mathcal{A}_tu}^2 = \sum_{Q \in \Box_t}\abs{Q}\abs{\fint_Q u}^2 \leq \sum_{Q \in \Box_t}\abs{Q}\fint_Q \abs{u}^2 = \norm{u}^2_2. &\qedhere
        \end{align}
    \end{proof}
    The following is an important tool for later that can be found in~\cite[Thm 4.3]{Morris}.
    We will repeat the notion of a \emph{Carleson measure} at the beginning of Section~\ref{sec:carleson}.
	\begin{proposition}[Carleson's inequality]
	    \label{prop:dyadic Carleson}
        If $\sigma$ is a Carleson measure on $\bO \times (0,1]$ with Carleson norm $\norm{\sigma}_{\mathcal{C}}$, then
        \begin{equation*}
            \iint_{\bO \times (0,1]}\abs{\mathcal{A}_tu(x)}^2 \d \sigma(x,t) \les \norm{\sigma}_{\mathcal{C}}\norm{u}^2_2 \, .
        \end{equation*}
	\end{proposition}

    \section{A coercivity property}
	\label{sec:coercivity_property}

    Throughout this section, we suppose that $\bO \subseteq \R^n$ is interior thick and that $\bD \subseteq \partial \bO$ is closed and Ahlfors--David regular. In addition, we suppose that $\bO$ is locally uniform near $\partial \bO \setminus \bD$. The goal of this section is to prove the following result.

    \begin{proposition}[Coercivity of $-\Delta_{\bAr{D}} + 1$]
		\label{ass:laplace}
		There exists $\alpha \in (0,1]$ and a constant $c>0$ such that for every function $v \in \dom(-\Delta_{\bAr{D}} + 1)$ there holds the estimate
		\begin{align}
			 \| (-\Delta_\bAr{D} + 1)^\frac{1+\alpha}{2} v \|_2 \geq c \| T_{\bAr{D}} v \|_{[\L^2(\bO)^m, \IW^{1,2}_\bAr{D}(\bO)]_\alpha}.
		\end{align}
	\end{proposition}

    In the Euclidean setting, that is to say, when $\bO = \R^n$, one can take $\alpha = 1$. Indeed, this follows by using the Fourier transform. In~\cite{AKM}, the authors introduce an abstract condition similar to Proposition~\ref{ass:laplace} which they call (H8). For the Kato square root problem on $\R^n$, it amounts to showing that the domain of $-\Delta$ coincides with $\W^{2,2}(\R^n)$ as well as that the Hessian of a function $u \in \dom(-\Delta)$ is controlled by $-\Delta u$ in norm. On domains, the analogue to Proposition~\ref{ass:laplace} is~\cite[(H7)]{Laplace-Extrapolation}.

    The proof of Proposition~\ref{ass:laplace} will ocupy the rest of this section.

    \subsection{Fractional Sobolev spaces}

    We will use fractional Sobolev spaces of functions that vanish on a subset $\bD \subseteq \partial \bO$. The starting point for their construction are the fractional Sobolev spaces $\W^{s,2}(\R^n)$, which can be defined either as Bessel potential spaces or as Sobolev--Slobodeckij spaces if $s$ is not an integer, see~\cite{Triebel}.

    Since $\bD$ is a $(n-1)$-set, see Remark~\ref{Rem: AD} for this terminology, a version of the Lebesgue differentiation theorem allows us to define traces on $\bD$. The following is a weakened version of \cite[Thm.~VI.1]{JW} that suffices for our purpose.

    \begin{proposition}
    \label{prop: JW}
    Let $s \in (\frac{1}{2}, \frac{3}{2})$ and $u \in \W^{s,2} ( \R^n ) $. For $\cH^{n-1}$-almost every $x \in \bD$ the limit
    \begin{align*}
    (\Res_{\bD} u)(x) \coloneqq \lim_{r\to 0} \frac{1}{|\B(x,r)|} \int_{\B(x,r)} u(y) \d y
    \end{align*}
    exists. The restriction operator $\Res_{\bD}$ maps $\W^{s,2}  (\R^n) $ boundedly into $\L^2(\bD, \cH^{n-1})$.
    \end{proposition}

    With the trace operator at hand, we introduce the closed subspace $\W^{s,2}_{\bD}(\R^n)$ of $\W^{s,2}(\R^n)$ by
    \begin{align}
    \W^{s,2}_{\bD}(\R^n) \coloneqq \bigl\{u\in \W^{s,2}(\R^n) \colon \Res_{\bD} u = 0 \bigr\}.
    \end{align}
    In the case $s=1$, this notion is consistent with Definition~\ref{Def: W12D}, see for instance~\cite[Lem.~3.3]{BE}. Finally, we denote the distributional restriction to $\bO$ by $|_{\bO}$ and define fractional Sobolev spaces on $\bO$ by restriction.

    \begin{definition}
    \label{Def: spaces on open sets}
    Let $s \in [0,\frac{3}{2})$ and $t \in (\frac{1}{2}, \frac{3}{2})$. Put $\W^{s,2}(\bO) \coloneqq \{ u|_{\bO}: u\in \W^{s,2}(\R^n) \}$ and $\W^{t,2}_{\bD}(\bO) \coloneqq \{ u|_{\bO}: u\in \W^{t,2}_{\bD}(\R^n) \}$ and equip them with quotient norms.
    \end{definition}

    \begin{remark}
    \label{Rem: spaces on open sets}
    These spaces are again Hilbert spaces by construction as quotients of Hilbert spaces. Since $\bO$ is interior thick, we have that $\W^{t,2}_{\bD}(\bO)$ is a closed subspace of $\W^{t,2}(\bO)$ with an equivalent norm~\cite[Lem.~3.4]{BE}. As a cautionary tale, let us stress that in the context of this section $\W^{1,2}(\bO)$ is embedded into but possibly \emph{not} equal to the collection of all $u \in \L^2(\bO)$ with $\nabla u \in \L^2(\bO)^d$ and norm \eqref{intrinsic Sobolev norm}. However, as a consequence of Theorem~\ref{Thm: W12D extension}, the definition of $\W^{1,2}_{\bD}(\bO)$ above coincides with the original one from Definition~\ref{Def: W12D} up to equivalent norms when working in the geometric setting of Theorem~\ref{thm:main}.
    \end{remark}

    The (fractional) Sobolev spaces respect the following interpolation rule. The result was originally shown in~\cite{BE} under the additional assumption that $\partial \bO$ is a $(n-1)$-set. Subsequently, the boundary regularity could be relaxed to porosity in~\cite[Prop.~4.16]{BEH-Kato}.
    A reader who is not yet familiar with abstract interpolation theory can find further information in~\cite{Triebel}.

    \begin{proposition}[Interpolation identity]
        \label{Prop: Interpolation}
        For $s\in (0,1)$, one has the following interpolation identity
        \begin{align}
        [\L^2(\bO), \W^{1,2}_{\bD}(\bO)]_s = \begin{cases}
        \W^{s,2}(\bO) &(\text{if $s < 1/2$}) \\
        \W_{\bD}^{s,2}(\bO) &(\text{if $s > 1/2$})
        \end{cases}.
        \end{align}
    \end{proposition}

    \subsection{Identification of fractional power domains of the Laplacian}

    We recall an identification result from~\cite{BEH-Kato} that concerns certain fractional powers of $-\Delta_\bAr{D} + 1$. This property relies on the geometric quality of $\bO$ and $\bAr{D}$ and invokes deep results from potential theory and related fields.

    In analogy with the case $\alpha=1$, we employ the notation $\IW^{\alpha,2}_\bAr{D}(\bO) \coloneqq \otimes_{i=1}^m \W^{\alpha,2}_{\bD_i}(\bO)$, where $\bAr{D} = (\bD_i)_{i=1}^m$ is a given array of Ahlfors--David regular boundary parts.

    \begin{theorem}
        \label{Thm: Laplace}
        Under the geometric assumptions fixed at the beginning of this section, there exists $\eps \in (0,\frac{1}{2})$ such that the fractional power domains of $-\Delta_{\bAr{D}} + 1$ in $\L^2(\bO)^m$ are given by
        \begin{align*}
        \dom((-\Delta_{\bAr{D}} + 1)^{\frac{\alpha}{2}})
        = \begin{cases}
        \IW^{\alpha,2}_{\bAr{D}}(\bO) & \text{if $\alpha \in (\frac{1}{2},1+\eps)$,} \\
        \IW^{\alpha,2}(\bO) & \text{if $\alpha \in (0,\frac{1}{2})$}.
        \end{cases}
        \end{align*}
        The identity holds with equivalence of norms, and, by invertibility of $-\Delta_{\bAr{D}} + 1$, the graph norm can be replaced by the homogeneous graph norm.
    \end{theorem}

    The proof of the preceding result can be found in~\cite[Thm.~1.2]{BEH-Kato}.

    \subsection{Proof of Proposition~\ref{ass:laplace}}

    Eventually, we can give the proof of the main result of this section.

    \begin{proof}[Proof of Proposition~\ref{ass:laplace}]
        Set $\alpha = \nicefrac{\eps}{2}$, where $\eps$ is introduced in Theorem~\ref{Thm: Laplace}. Note that $\alpha < \nicefrac{1}{2}$.
        Now, let $v \in \dom(-\Delta_\bAr{D} + 1)$. Using Proposition~\ref{Prop: Interpolation} and the definition of $T = T_\bAr{D}$, calculate
        \begin{align}
            \| T v \|_{[\L^2(\bO)^m, \IW^{1,2}_\bAr{D}(\bO)]_\alpha}^2 \approx \| T v \|_{\IW^{\alpha,2}_\bAr{D}(\bO)}^2
            = \| v \|_{\IW^{\alpha,2}_\bAr{D}(\bO)}^2 + \| \nabla v \|_{\IW^{\alpha,2}_\bAr{D}(\bO)}^2.
        \end{align}
        By definition of the scale of spaces $\W^{s,2}(\bO)$, the gradient maps $\W^{1+\alpha,2}(\bO) \to \W^{\alpha,2}(\bO)$ (see~\cite[Lem.~4.10]{BEH-Kato} for details). Hence, taking Remark~\ref{Rem: spaces on open sets} and $1 + \alpha \in [1, 1 + \nicefrac{1}{2})$ into account, we continue by
        \begin{align}
            \| T v \|_{[\L^2(\bO)^m, \IW^{1,2}_{\bAr{D}}(\bO)]_\alpha} \lesssim \| v \|_{\IW^{1+\alpha,2}(\bO)} \approx \| v \|_{\IW^{1+\alpha,2}_\bAr{D}(\bO)}.
        \end{align}
        Eventually, Theorem~\ref{Thm: Laplace} lets us conclude
        \begin{align}
            \| T v \|_{[\L^2(\bO)^m, \IW^{1,2}_\bAr{D}(\bO)]_\alpha} \lesssim \| (-\Delta_\bAr{D} + 1)^\frac{1+\alpha}{2} v \|_2. &\qedhere
        \end{align}
    \end{proof}

	\section{Reduction to a quadratic estimate}
	\label{sec:reduction}

    In Sections~\ref{sec:reduction}-\ref{sec:carleson}, we are going to prove our main result, Theorem~\ref{thm:main}, under the additional assumption that $\bO$ is interior thick. This statement is equivalent to the validity of certain quadratic estimates. The argument is well-known in the literature (see~\cite[Chap. 13]{ISEM} for instance), but we are going to briefly recap on it in this section for the reader's convenience. Eventually, we will remove the interior thickness condition in Section~\ref{sec:no_thickness} to complete the proof of Theorem~\ref{thm:main}.
    
    The contents of this section are still valid without the interior thickness condition, so we will not yet employ the bold letter convention (Convention~\ref{conv:bold}) here. However, interior thickness will be crucial in the verification of the quadratic estimate~\eqref{eq:QE} appearing in the following proposition.

	\begin{proposition}[Reduction to a quadratic estimate]
    \label{prop:reduc to QE}
		In the situation of Theorem~\ref{thm:main}, the validity of Kato's square root property is equivalent to the quadratic estimate
		\begin{align}
			\label{eq:QE}
			\tag{QE}
			\int_0^\infty \| tL (1 + t^2 L)^{-1} u\|_2^2 \ddt \lesssim \norm{ u }^2_{\IW^{1,2}(O)}, \qquad u \in \IW^{1,2}_\Ar{D}(O),
		\end{align}
		for every elliptic system in divergence form $L$ on $O$ subject to Dirichlet boundary conditions in $\Ar{D} \subseteq (\bd O)^m$
	\end{proposition}

	Hence, instead of showing Theorem~\ref{thm:main} right away, we are going to establish~\eqref{eq:QE} in the following two sections.

    As in~\cite{ISEM}, the quadratic estimate~\eqref{eq:QE} can be reformulated in terms of the bounded operator $ \mathcal{L} : \IW^{1,2}_\Ar{D} (O) \to ( \IW^{1,2}_\Ar{D} (O) )^* $.
    More precisely, writing
    \begin{equation*}
        \Theta_t = t ( 1+ t^2 \mathcal{L} )^{-1} T_\Ar{D}^* B
    \end{equation*}
    with $T_D^*$ in the \enquote{very weak formulation} and keeping $\norm{ u }^2_{\IW^{1,2}(O)} = \normm{T_\Ar{D} u}_2^2$ in mind,~\eqref{eq:QE}~takes the form
    \begin{align}
    \label{eq:Theta_t QE}
        \int_0^\infty \norm{ \Theta_t T_\Ar{D} u }^2_2
        \ddt \les \norm{T_\Ar{D} u}_2^2 \, ,
        \qquad u \in \IW^{1,2}_\Ar{D} (O) \, .
    \end{align}

    \begin{lemma}
        \label{lemma:Theta odes}
        The family $ \{ \Theta_t \}_{ t>0 } $ is uniformly bounded in $ \mathcal{L} ( \L^2 (O)^N , \L^2 (O)^m ) $ and satisfies $\L^2$ off-diagonal estimates.
    \end{lemma}

    \begin{proof}
        Since the coefficient matrix $ B $ acts as a bounded multiplication operator on $ \L^2 (O)^N $ and preserves the support, the claim is a direct consequence of Proposition~\ref{prop:ode}.
    \end{proof}

	\section{Principal part approximation}
	\label{sec:principal part approximation}
	
	Throughout this section, we suppose that the pair $(\bO, \bAr{D})$ satisfies the geometric assumptions from Theorem~\ref{thm:main} as well as that $\bO$ is interior thick. Recall Convention~\ref{conv:bold} on the usage of bold symbols. The goal of this section is to reduce the validity of~\eqref{eq:QE} to a Carleson measure estimate. The latter will be verified in Section~\ref{sec:carleson}.

	\subsection{Presentation of the relevant operators}
	\label{sec:principal part definitions}

	It is well-known that, owing to the off-diagonal estimates from Lemma~\ref{lemma:Theta odes}, $\bTheta_t$ extends to a bounded operator $\L^\infty(\bO)^N \to \L^2_{\loc}(\bO)^m $ for all $t$, see for instance~\cite[Prop.~11.12 \& Def.~11.14]{ISEM}. We exploit this fact in the subsequent definition of the principal part.

	\begin{definition}[Principal part]
		\label{def:principal part}
        Identifying the standard unit vectors $ e_1 , e_2, \dots, e_N \in \C^N $ with their respective constant functions on $\bO$, we define for $t \in (0,1]$ the \emph{principal part} of $\bTheta_t$ by
            \begin{align}
            \bgamma_t := ( \bTheta_t(e_1) , \ldots , \bTheta_t ( e_N) )
            \in \L^2_{\mathrm{loc}}(\bO , \C^{ m \times N})
            \end{align}
	\end{definition}
    \begin{remark}
        \label{rem:def of gamma}
        For almost every $ x \in \bO $, the principal part $ \bgamma_t (x) $ is a linear map from $ \C^N $ to $ \C^m $ acting as
        \begin{align}
            \bgamma_t (x) w = \bTheta_t ( w \one_{  \R^n } ) (x)
            \, , \quad w \in \C^N \, .
        \end{align}
        In particular, for $ F \in \L^2 ( \bO )^N $, $ t \in (0,1] $ and a dyadic cube $ Q \in \Box_t $ we have
        \begin{align}
            \label{eq:formula for gamma A}
            \one_Q ( \bTheta_t - \bgamma_t \mathcal{A}_t ) F
            = \one_Q \bTheta_t ( F - (F)_Q ) \, .
        \end{align}
    \end{remark}

	\begin{definition}[Smoothing operators]
		\label{def:smooting operators}
        For $ t>0 $, we define the smoothing operators
        \begin{equation*}
            \bP_t = ( 1+ t^2 M_\bAr{D} )^{-1} \in \mathcal{L}
            ( \L^2 ( \bO )^N ) \ ,
            \quad
            \bQ_t = t T^*_\bAr{D} \bP_t
            \in \mathcal{L} ( \L^2 ( \bO )^N ,
            \L^2 ( \bO )^m ) \, .
        \end{equation*}
	\end{definition}

   A direct consequence of Proposition~\ref{prop:M and Delta} with the function $ f : z \to ( 1+ t^2 z )^{-1} $ is the commuting relation
\begin{equation}
    \label{eq:intertwining of res}
    \bP_t T_\bAr{D} = T_\bAr{D} ( 1 + t^2 (-\Delta_\bAr{D} + 1 ) )^{-1} \, .
\end{equation}

    \begin{lemma}
        \label{lemma:properties of smoothing operators}
        The families $ \{ \bP_t \} $ and $ \{ \bQ_t \} $ are uniformly bounded in $ \mathcal{L} ( \L^2 ( \bO )^N ) $ and $ \mathcal{L} ( \L^2 ( \bO )^N , \L^2 ( \bO )^m ) $, respectively.
        Thus, also the adjoint family $ \{ \bQ_t^* \} $ is uniformly bounded in $ \mathcal{L} ( \L^2 ( \bO )^m , \L^2 ( \bO )^N ) $, and $ \bQ_t^* $ is given by the bounded extension of $ t \bP_t T_\bAr{D} $ to $ \L^2 ( \bO )^m $.
    \end{lemma}

    \begin{proof}
        Uniform boundedness of $ \{ \bP_t \} $ is by $ m $-accretivity of $ M_\bAr{D} $,
        and for $ \{ \bQ_t \} $ one can use the sesquilinear form $\boldsymbol{\mathfrak{b}}$, see the proof of Lemma~\ref{lem:M perturbation} or~\cite[Lem.~11.7]{ISEM}.
    \end{proof}

    \begin{lemma}
        \label{lemma:qe for Qt}
        For $ F \in \L^2 ( \bO )^N $, one has
        \begin{equation*}
            \int_0^\infty \norm{ \bQ_t F }^2_2
            \ddt \les
            \norm{F}^2_2 \, .
        \end{equation*}
        If in addition $ F \in \ran(T_\bAr{D}) $, then we have the reproducing formula
        \begin{equation*}
            2 \int_0^\infty \bQ_t^* \bQ_t F \ddt = F \, .
        \end{equation*}
    \end{lemma}
    \begin{proof}
        Owing to Proposition~\ref{prop: M is self-adjoint}, the square root property for the self-adjoint operator $ M_\bAr{D} $ gives
        \begin{equation*}
            \norm{ \bQ_t F }_2 = \normm{ t T^*_\bAr{D}
            (1+t^2  M_\bAr{D} )^{-1} F }_2
            = \normm{ \sqrt{ t^2 M_\bAr{D} }
            (1+t^2 M_\bAr{D} )^{-1} F }_2 \, .
        \end{equation*}
        Hence, the first assertion follows by quadratic estimates for $ M_\bAr{D} $.

        For the reproducing formula, note that $ \bQ_t $ maps into $ \dom ( T_\bAr{D} ) $, so we can compute
        \begin{equation*}
            \bQ_t^* \bQ_t
            = t^2 (1+t^2 M_\bAr{D} )^{-1} T_\bAr{D} T^*_\bAr{D}
            (1+t^2  M_\bAr{D} )^{-1}
            = t^2 M_\bAr{D} (1+t^2  M_\bAr{D} )^{-2}
            = f( t^2 M_\bAr{D} ) \, ,
        \end{equation*}
        where $ f (z) = z (1+z)^{-2} $.
        By Remark~\ref{rem: T* is surj} we have $ \ran(T_\bAr{D}) = \ran(M_\bAr{D}) $,
        so the Calderón reproducing formula~\cite[Thm.~6.16]{ISEM} yields the assertion.
    \end{proof}

    With these operators, we first split the left-hand side of~\eqref{eq:QE} as
    \begin{equation}
        \label{eq:split quadratic estimates integral}
        \int_0^\infty \norm{ \bTheta_t T_\bAr{D} u }^2_2 \ddt
        \les \int_0^\infty \norm{ \bTheta_t \bP_t T_\bAr{D} u }^2_2
        \ddt + \int_0^\infty \norm{ \bTheta_t
        ( 1- \bP_t ) T_\bAr{D} u }^2_2 \ddt \, .
    \end{equation}

    The second term on the right-hand side can be controlled via a simple algebraic manipulation.

    \begin{proposition}
        \label{prop:QE for Theta (1-P) nabla u}
        For $ F \in \L^2 ( \bO )^N $, there holds the quadratic estimate
        \begin{equation*}
            \int_0^\infty \norm{ \bTheta_t ( 1- \bP_t ) F }^2_2
            \ddt \les \norm{F}^2_2 \, .
        \end{equation*}
    \end{proposition}
    \begin{proof}
        Note that $ 1-\bP_t = t^2 M_\bAr{D} \bP_t $ and $ T^*_\bAr{D}  \bB M_\bAr{D} = T^*_\bAr{D}  \bB T_\bAr{D} T^*_\bAr{D}  \subseteq \bcL T^*_\bAr{D}  $.
        Thus, for $ F \in \L^2 ( \bO )^N $, we can compute
    \begin{align*}
        \bTheta_t (1- \bP_t ) F
        &= t (1+ t^2 \bcL )^{-1}
        T^*_\bAr{D}  \bB t^2 M_\bAr{D} \bP_t F \\
        &= t^2 (1+ t^2 \bcL )^{-1}
        \bcL t T^*_\bAr{D}  \bP_t F \\
        &= t^2 \bL (1+t^2 \bL )^{-1} \bQ_t F
        = ( 1 - (1+t^2 \bL )^{-1} ) \bQ_t F \, .
    \end{align*}
    By $ m $-accretivity, the resolvents $ ( 1+ t^2 \bL )^{-1} $ are uniformly bounded, so we arrive at
    \begin{equation*}
        \int_0^\infty \norm{  \bTheta_t (1- \bP_t ) F}^2_2
        \ddt
        \les \int_0^\infty \norm{ \bQ_t F }^2_2
        \ddt \les \norm{F}^2_2 \, ,
    \end{equation*}
    where the last step follows from the quadratic estimate for $ \bQ_t $ from Lemma~\ref{lemma:qe for Qt}.
    \end{proof}

	\subsection{Reduction to finite time}
	\label{sec:reduction finite time}

    We want to approximate the first term in~\eqref{eq:split quadratic estimates integral} by the principal part. However, since the dyadic structure on $\bO$ and hence the approximation $ \bgamma_t \mathcal{A}_t $ only exist on small scales,
   	the integral for $ t \geq 1 $ has to be treated separately.

    \begin{proposition}
        \label{prop:reduction to finite time}
        It holds
        \begin{equation*}
            \int_1^\infty
            \norm{ \bTheta_t \bP_t T_\bAr{D} u }^2_2
            \ddt \les \norm{ T_\bAr{D} u }^2_2
            \, , \qquad u \in \IW^{1,2}_\bAr{D}( \bO ) \, .
        \end{equation*}
    \end{proposition}

    \begin{proof}
        Denote by $ H_\alpha $ the complex interpolation space $ [ \L^2 ( \bO )^m , \IW^{1,2}_\bAr{D} ( \bO ) ]_\alpha $.
        Since $ \bTheta_t $ is uniformly bounded in $ \mathcal{L} ( \L^2 ( \bO )^N , \L^2 ( \bO )^m ) $ and $ \H_\alpha \embeds \L^2 ( \bO )^m $, we have
        \begin{equation*}
            \norm{ \bTheta_t \bP_t T_\bAr{D} u }_2
            \les \norm{ \bP_t T_\bAr{D} u }_2
            \les \norm{ \bP_t T_\bAr{D} u }
            _{ H_\alpha }.
        \end{equation*}
        Recall from~\eqref{eq:intertwining of res} the commuting relation $ \bP_t T_\bAr{D} u = T_\bAr{D} (1 + t^2 (-\Delta_\bAr{D} + 1) )^{-1} u $.
        Since the resolvent $ (1 + t^2 (-\Delta_\bAr{D} + 1 ) )^{-1} $ maps into $ \dom (-\Delta_\bAr{D}+1) $ we can use Proposition~\ref{ass:laplace} in conjunction with the square root property of $ -\Delta_\bAr{D} + 1 $ and the intertwining property from Proposition~\ref{prop:M and Delta} with the $ \H^\infty $-function $ f(z) = \frac{ z^\alpha }{ 1+t^2 z } $ to obtain
        \begin{align*}
            \norm{ \bP_t T_\bAr{D} u }_{ H_\alpha }
            &={} \normm{ T_\bAr{D} (1 + t^2 (-\Delta_\bAr{D} + 1 ) )^{-1} u }_{ H_\alpha } \\
            &\les{} \normm{ ( -\Delta_\bAr{D} + 1 )^{ \frac{ 1+ \alpha }{2} } ( 1 + t^2 (-\Delta_\bAr{D} + 1 ) )^{-1} u }_2 \\
            &={} \normm{ T_\bAr{D} ( -\Delta_\bAr{D} + 1 )^\frac{\alpha}{2}
            (1 + t^2 (-\Delta_\bAr{D} + 1 ) )^{-1} u }_2 \\
            &={}     \normm{ (M_\bAr{D})^\frac{\alpha}{2} \bP_t T_\bAr{D} u }_2.
        \end{align*}
        We can then estimate $ 1 \leq t^{ 2 \alpha } $ to obtain a square function estimate for $ M_\bAr{D} $ with $ f(z) = z^{ \alpha /2 } ( 1+ z)^{-1} $, which reads
        \begin{align*}
            \int_1^\infty
            \normm{ \bTheta_t \bP_t T_\bAr{D} u }^2_2 \ddt
            &\les \int_1^\infty t^{2 \alpha }
            \normm{ (M_\bAr{D})^\frac{\alpha}{2} ( 1+t^2 M_\bAr{D} )^{-1} T_\bAr{D} u }^2_2 \ddt \\
            &= \int_1^\infty \normm{ f( t^2 M_\bAr{D} )
            T_\bAr{D} u }^2_2 \ddt
            \les \normm{ T_\bAr{D} u }^2_2 \, ,
        \end{align*}
    	where we used Proposition~\ref{prop: M is self-adjoint} in the final step.
    \end{proof}

	\subsection{Smoothed principal part approximation}
	\label{sec:smoothed}

    As the next step, we want to approximate $\bTheta_t \bP_t T_\bAr{D}$ by the smoothed principle part $\bgamma_t \mathcal{A}_t \bP_t T_\bAr{D}$. This will be the content of Proposition~\ref{prop:smooth principal part}. As a preparation, we establish the following lemma first.
	To this end, recall from Remark~\ref{rem:def of gamma} the (pointwise) matrix action of $\bgamma_t$.

    \begin{lemma}
        \label{lemma:13.6 in lectures}
        The family $\bgamma_t \mathcal{A}_t$ is uniformly bounded in $\mathcal{L}(\L^2(\bO)^N)$ for $t\in (0,1]$. Moreover, for $ F = (F_0 , \ldots , F_n ) \in \IW^{1,2}_\bAr{D} ( \bO )^{n+1} $, there holds
        \begin{equation*}
         \norm{ ( \bTheta_t - \bgamma_t \mathcal{A}_t ) F }_2
         \les t \norm{T_\bAr{D} F}_2
         , \quad t \in (0,1] .
        \end{equation*}
    \end{lemma}
    \begin{proof}
    	Fix $t\in (0,1]$ and the natural number $k$ for which $t \in (2^{-(k+1)}, 2^{-k}]$. First, we work on a fixed $Q \in \Box_t$. By utilizing Proposition~\ref{prop:dyadic structure} $(4)$, there exist constants $a_0, a_1 > 0$ and $z \in \bO$ such that
        \begin{align}
        \label{eq:nesting_dyadic_cube}
            \Q(z, a_0 2^{-k}) \cap \bO \subseteq Q \subseteq \Q(z, a_1 2^{-k}) \cap \bO.
        \end{align}
        Let $Q_* \coloneqq \Q(z, a_0 2^{-k})$ and $Q^* \coloneqq \Q(z, a_1 2^{-k})$. Then, the left-hand side of~\eqref{eq:nesting_dyadic_cube} becomes $Q_* \cap \bO$ and the right-hand side $Q^* \cap \bO$.
        Owing to interior thickness of $\bO$, there holds
        \begin{align}
            \label{eq:Q_interior_thickness}
            \abs{Q} \gtrsim \abs{Q_*}.
        \end{align}
        For a bounded set $K \subseteq \bO$ and $b \in\L^\infty(\bO)^N$, the limit $\bTheta_t b := \lim_{k \to \infty} \bTheta_t(\mathbf{1}_{2^k Q^* \cap \bO}b)$ exists in $\L^2(K)^N$ and is independent of $Q^*$. Indeed, this is a consequence of the off-diagonal estimates of $\bTheta_t$, see for instance~\cite[Prop.~11.12]{ISEM}.
        We have already employed the same construction when defining the principal part in Definition~\ref{def:principal part}.
        Now, we define \emph{annuli} of a \enquote{cube} $Q$ by
    \begin{equation*}
        C_1(Q) \coloneqq 4Q^* \cap \bO, \quad C_{\ell}(Q) \coloneqq \Bigl(2^{\ell + 1}Q^* \setminus 2^{\ell}Q^* \Bigr) \cap \bO,
    \end{equation*}
    where $\ell \geq 2$.
    Using this notation in the definition of $\bTheta_t b$ yields
    \begin{align*}
        \bgamma_t \mathcal{A}_t F &= \underset{k \to \infty}{\lim} \bTheta_t(\mathbf{1}_{2^kQ^*\cap \bO}(F)_Q)
        = \underset{k \to \infty}{\lim} \sum_{\ell = 1}^{k-1} \bTheta_t(\mathbf{1}_{C_{\ell}(Q)}(F)_Q)
        = \sum_{\ell =1}^{\infty} \bTheta_t(\mathbf{1}_{C_{\ell}(Q)}(F)_Q).
    \end{align*}
    Now, note for $\ell \geq 2$ that
    \begin{equation*}
        \dfrac{\mathrm{d}(C_\ell(Q),Q)}{t} \geq \dfrac{\mathrm{d}(C_\ell(Q),Q^*)}{t} \gtrsim 2^\ell. %
    \end{equation*}
    Both of the above, in conjunction with Lemma~\ref{lemma:Theta odes}, allow us to see that
    \begin{align*}
        \norm{\mathbf{1}_Q (\bgamma_t \mathcal{A}_t F)}_2 &\leq \sum_{\ell = 1}^\infty\norm{\mathbf{1}_Q\bTheta_t(\mathbf{1}_{C_{\ell}(Q)}(F)_Q)}_2 \\
        &\les \sum_{\ell = 1}^\infty e^{-c 2^\ell}\norm{\mathbf{1}_{C_{\ell}(Q)}(F)_Q}_2.
    \end{align*}
    By H\"older's inequality and~\eqref{eq:Q_interior_thickness}, we have
    \begin{align}
        \norm{\mathbf{1}_{C_\ell(Q)}(F)_Q}_2 &= \dfrac{\abs{C_\ell(Q)}^\frac{1}{2}}{\abs{Q}}\abs{\int_Q F(y)\d y }
        \leq \left( \dfrac{\abs{C_\ell(Q)}}{\abs{Q}}\right) ^\frac{1}{2}\norm{\mathbf{1}_Q F}_2
        \lesssim 2^\frac{\ell n}{2} \norm{\mathbf{1}_Q F}_2,
    \end{align}
    which gives us
    \begin{equation}
        \norm{\mathbf{1}_Q (\bgamma_t \mathcal{A}_t F)}_2 \les \left(\sum_{\ell = 1}^\infty 2^{\frac{\ell n}{2}} e^{-c 2^\ell} \right)\norm{\mathbf{1}_Q F}_2 \lesssim \norm{\mathbf{1}_Q F}_2.
    \end{equation}
    Since the implicit constant is independent of $Q$ and $t$, squaring the last bound and summing in $Q$ yields the first assertion. Now, we can proceed with the second assertion. Since $\bTheta_t$ is bounded on $\L^2(\bO)^N$, we have
    \begin{equation*}
        \bTheta_tF = \sum_{\ell =1}^{\infty}\bTheta_t(\mathbf{1}_{C_{\ell}(Q)}F).
    \end{equation*}
    Taking the difference with $\gamma_t \mathcal{A}_t F$ and again applying off-diagonal estimates, %
    we have
    \begin{align}
        \label{eq:sum_for_local_poincare}
        \norm{\mathbf{1}_{Q}(\bTheta_t - \bgamma_t \mathcal{A}_t)F}_2 &\leq \sum_{\ell =1}^{\infty}\norm{\mathbf{1}_{Q}\bTheta_t(\mathbf{1}_{C_\ell(Q)}(F-(F)_Q))}_2 \\
        &\les \sum_{\ell =1}^{\infty}e^{-c 2^\ell}\norm{\mathbf{1}_{2^{\ell+1}Q^*\cap\bO}(F-(F)_Q)}_2.
    \end{align}
    Now, we distinguish whether $2^\ell t$ is small or large. More precisely, there is a natural number $\ell_0$ such that the local Poincaré inequality from Remark~\ref{rem:local_poincare} is applicable on $2^\ell Q^* \cap \bO$ if $\ell \leq \ell_0$ (note that we tacitly use comparability between balls and cubes here). To the contrary, if $\ell > \ell_0$, we deduce that $2^\ell t \geq \kappa$, where $\kappa$ is a constant depending on the geometry.
    In the first case, fix $\ell \leq \ell_0$ and deduce by virtue of the local Poincaré inequality (Remark~\ref{rem:local_poincare}) the estimate
    \begin{align}
        \norm{\mathbf{1}_{2^{\ell+1}Q^*\cap\bO}(F-(F)_Q)}_2 \lesssim \ell 2^{\nicefrac{n \ell}{2}} 2^\ell t \| \nabla F \|_{\L^2(c 2^{\ell+1}Q^*\cap\bO)} \leq \ell 2^{\nicefrac{n \ell}{2}} 2^\ell t \| T_\bAr{D} F \|_{\L^2(c 2^{\ell+1}Q^*\cap\bO)}.
    \end{align}
    In the second case, for fixed $\ell > \ell_0$, estimate using Jensen's inequality
    \begin{align}
        \norm{\mathbf{1}_{2^{\ell+1}Q^*\cap\bO}(F-(F)_Q)}_2 \lesssim 2^{\nicefrac{n\ell}{2}} \| F \|_{\L^2(2^{\ell+1}Q^*\cap\bO)} \lesssim 2^{\nicefrac{n\ell}{2}} 2^\ell t \| T_\bAr{D} F \|_{\L^2(2^{\ell+1}Q^*\cap\bO)}.
    \end{align}
    Plugging the last two bounds back into~\eqref{eq:sum_for_local_poincare} and using the double exponential decay, we obtain
    \begin{align}
        \norm{\mathbf{1}_{Q}(\bTheta_t - \bgamma_t \mathcal{A}_t)F}_2 \lesssim t \sum_{\ell=1}^\infty e^{-c 2^\ell} \| T_\bAr{D} F \|_{\L^2(c 2^{\ell+1}Q^*\cap\bO)}.
    \end{align}
    To conclude the proof, we have to sum the last bound over all dyadic \enquote{cubes} on scale $t$. Due to the rapid decay in $\ell$, this follows by standard arguments, see for instance~\cite[Lem.~13.6]{ISEM} for details.
    \end{proof}

    \begin{remark}
        Instead of using a local Poincaré inequality on small scales, there is another approach to obtain the aforementioned result in the literature based on a weighted (global) Poincaré inequality, see~\cite[Prop.~8.2]{Laplace-Extrapolation}. For that approach, the existence of an (inhomogeneous) Sobolev extension operator is necessary. Even though our local Poincaré inequality was deduced by using a local and homogeneous Sobolev extension operator, there might be other settings in which local Poincaré inequalities are easier to obtain than a Sobolev extension operator.
    \end{remark}

    \begin{proposition}
    \label{prop:smooth principal part}
        There holds the quadratic estimate
        \begin{equation*}
            \int_0^1 \norm{ ( \bTheta_t - \bgamma_t \mathcal{A}_t )
            \bP_t T_\bAr{D} u }_2^2 \ddt
            \les \norm{ T_\bAr{D} u }_2^2 \, ,
            \qquad u \in \IW^{1,2}_\bAr{D} ( \bO ) \, .
        \end{equation*}
    \end{proposition}
    \begin{proof}
        On the one hand, $\bTheta_t - \bgamma_t \A_t$ is uniformly bounded in $\mathcal{L}(\L^2(\bO)^N, \L^2(\bO)^m )$, see Lemma~\ref{lemma:Theta odes} and Lemma~\ref{lemma:13.6 in lectures}. On the other hand, the $\mathcal{L}( \IW^{1,2}_\bAr{D} ( \bO )^{n+1}, \L^2 ( \bO )^m ) $-norm of $\bTheta_t - \bgamma_t \A_t$ is controlled by $t$ by virtue of Lemma~\ref{lemma:13.6 in lectures}. By interpolation, we obtain boundedness $$\bTheta_t - \bgamma_t \A_t : (H_\alpha)^{n+1} \to \L^2(\bO)^m , \qquad \| \bTheta_t - \bgamma_t \A_t \|_{ H_\alpha \to \L^2(\bO) } \les t^\alpha \, , $$
        where $ H_\alpha $ denotes again the interpolation space $ [ \L^2 ( \bO )^m , \IW^{1,2}_\bAr{D} ( \bO ) ]_\alpha $.
        Therefore, in conjunction with the intertwining property and Proposition~\ref{ass:laplace} (compare with the calculation in the proof of Proposition~\ref{prop:reduction to finite time}), we obtain that
        \begin{equation*}
            \norm{ ( \bTheta_t - \bgamma_t \mathcal{A}_t )
            \bP_t T_\bAr{D} u }_2
            \les t^\alpha \norm{ \bP_t T_\bAr{D} u }_{ H^\alpha }
            \les t^\alpha
            \normm{ (M_\bAr{D})^{ \frac{\alpha}{2} } \bP_t T_\bAr{D} u }_2 .
        \end{equation*}
        With the $ \H^\infty $-function $ f : z \mapsto z^{ \alpha /2 } (1+z)^{-1} $, quadratic estimates for $ M_\bAr{D} $ give
        \begin{equation*}
            \int_0^1 \normm{ ( \bTheta_t - \bgamma_t \mathcal{A}_t )
            \bP_t T_\bAr{D} u }_2^2 \ddt
            \les
            \int_0^1 \normm{ f( t^2 M_\bAr{D} ) T_\bAr{D} u }_2^2 \ddt
            \les \norm{ T_\bAr{D} u }_2^2 \, .
            \qedhere
        \end{equation*}
    \end{proof}

	\subsection{Reduction to a Carleson measure estimate}
	\label{sec:remove smooth}

    So far, we have reduced~\eqref{eq:QE} to the quadratic estimate
    \begin{equation*}
        \int_0^1 \norm{ \bgamma_t \mathcal{A}_t
        \bP_t T_\bAr{D} u }^2_2 \ddt
        \les \norm{ T_\bAr{D} u }^2_2 \, ,
        \qquad u \in \IW^{1,2}_\bAr{D} ( \bO ) \, .
    \end{equation*}
    To eliminate the smoothing via the resolvents $ \bP_t $ and to reduce this to a Carleson measure estimate, we split again
    \begin{equation}
        \label{eq:remove smoothing}
        \int_0^1 \norm{ \bgamma_t \mathcal{A}_t
        \bP_t T_\bAr{D} u }^2_2 \ddt \les
        \int_0^1 \norm{ \bgamma_t \mathcal{A}_t
        ( \bP_t -1 ) T_\bAr{D} u }^2_2 \ddt +
        \int_0^1 \norm{ \bgamma_t \mathcal{A}_t
        T_\bAr{D} u }^2_2 \ddt .
    \end{equation}
    The second term can be dominated by
    \begin{equation*}
        \int_0^1 \int_{ \R^d } \abs{ \bgamma_t (x) (\mathcal{A}_t
        T_\bAr{D} u) (x) }^2 \, \frac{ \d x \d t }{t}
        \leq \int_0^1 \int_{ \R^d } \abs{ ( \mathcal{A}_t T_\bAr{D} u ) (x) }^2
        \, \norm{ \bgamma_t (x) }^2 \frac{ \d x \d t }{t} \, ,
    \end{equation*}
    where $ \norm{ \bgamma_t (x) } $ denotes the operator norm of $ \bgamma_t (x) \in \mathcal{L} ( \C^N , \C^m ) $.
    By Proposition~\ref{prop:dyadic Carleson}, this integral is controlled by $ \normm{ T_\bAr{D} u }^2_2 $, provided $ \d \sigma = \norm{ \bgamma_t }^2 \tfrac{ \d x \d t }{t} $ is a Carleson measure, which we are going to prove in Section~\ref{sec:carleson}.
    To treat the first term on the right-hand side of~\eqref{eq:remove smoothing}, we need the following interpolation inequality, which will be useful again in Section~\ref{sec:carleson}.

    Recall the notation $Tu = (u, \nabla u)^t$ for $u \in \IW^{1,2}(\bO)$. Later, we apply the subsequent results with $u \in \IW^{1,2}_\bAr{D}(\bO)$, so that $Tu = T_\bAr{D} u$ as usual.

    \begin{lemma}[Interpolation inequality]
    	\label{lem:interp ineq}
    	Let $\eta$ be as in Proposition~\ref{prop:dyadic structure}. We have the bound
    	\begin{align*}
    		\abs{\fint_Q \nabla u}^2 \les \dfrac{1}{t^\eta} \left(\fint_Q\abs{u}^2\right)^{\frac{\eta}{2}}\left(\fint_Q\abs{\nabla u}^2\right)^{1-\frac{\eta}{2}}%
    	\end{align*}
    	for all $t \in (0,1]$, $Q \in \Box_t$ and $u \in \IW^{1,2}(\bO)$.
    \end{lemma}
    \begin{proof}
        Fix $t\in (0,1]$, $Q \in \Box_t$ and $u \in \IW^{1,2}(\bO)$. Write the estimate above as
        \begin{equation*}
            X \les t^{-\eta}Y^{\eta/2}Z^{1-\eta/2}. %
        \end{equation*}
        First, observe that $X \leq Z$ by Jensen's inequality and that $Y=0$ implies $Z=0$. Thus, we can safely assume $Y,Z > 0$. Put $\tau := Y^{1/2}Z^{-1/2} > 0$ and note that, for $\tau\geq t$, we have
        \begin{equation*}
            X \leq Z \leq \tau^{\eta}t^{-\eta}Z = t^{-\eta}Y^{\eta/2}Z^{1-\eta/2}.
        \end{equation*}
        Hence, we can suppose $ \tau < t $, which circumvents awkward cases. Let $Q_r = \{x \in Q:\ \d(x,\R^n \setminus Q) \leq r\}$ for $r>0$, and note that, by Proposition~\ref{prop:dyadic structure} (5), we have the estimate
        \begin{align}
        \label{eq:cube_bdr}
            \abs{Q_r} \leq Cr^\eta \ell(Q)^{-\eta}\abs{Q} \leq  Cr^\eta t^{-\eta}\abs{Q}.
        \end{align}
        Convolve $\mathbf{1}_{Q \setminus Q_{\tau/2}}$ by a suitable kernel to obtain $\varphi \in \Cont_c^{\infty}(Q)$ with range in $[0,1]$, equal to 1 on $Q \setminus Q_\tau$ and satisfying $\norm{\nabla \varphi}_\infty \leq \frac{c}{\tau}$ for $c >0 $ depending only on $n$.
        Now, write
        \begin{equation*}
            \nabla u = \varphi \nabla u + (1-\varphi)\nabla u,
        \end{equation*}
        and apply integration by parts (here, we use that $\varphi$ is compactly supported in $Q$) to see that
        \begin{equation*}
            \int_Q \nabla u = -\int_Q u\nabla \varphi + \int_Q (1-\varphi)\nabla u.
        \end{equation*}
        From here, apply the inequality $(x+y)^2 \leq 2(x^2+y^2)$ as well as H\"older's inequality, noting that both $\nabla \varphi$ and $1-\varphi$ vanish on $Q \setminus Q_\tau$ and using~\eqref{eq:cube_bdr}, to achieve
        \begin{align*}
            \abs{\int_Q \nabla u}^2 &\leq 2\left( \int_Q \mathbf{1}_{Q_\tau} \abs{\nabla \varphi}^2\right)\left(\int_Q\abs{u}^2 \right) + 2\left( \int_Q \mathbf{1}_{Q_\tau} \abs{1- \varphi}^2\right)\left(\int_Q\abs{\nabla u}^2 \right) \\
            &\les \abs{Q_\tau}\left(\tau^{-2}\int_Q\abs{u}^2 + \int_Q\abs{\nabla u}^2\right) \\
            &\les \tau^\eta t^{-\eta} \abs{Q}\left(\tau^{-2}\int_Q\abs{u}^2 + \int_Q\abs{\nabla u}^2\right).
        \end{align*}
        Translating back in terms of $X$, $Y$ and $Z$, the above says
        \begin{align*}
            \abs{Q}^2 X &\les \tau^{\eta-2} t^{-\eta} \abs{Q}^2Y + \tau^\eta t^{-\eta}\abs{Q}^2 Z \\
            &= 2\abs{Q}^2t^{-\eta}Y^{\eta/2}Z^{1-\eta/2}.
        \end{align*}
        Now, the claim follows when dividing the last inequality by $\abs{Q}^2$. %
    \end{proof}
    The following inhomogeneous version of Lemma~\ref{lem:interp ineq} is an easy consequence.

    \begin{corollary}
        \label{cor:inhom interp ineq}
        With $ \eta $ as in Proposition~\ref{prop:dyadic structure}, we have
        \begin{equation}
            \abs{ \fint_Q T u }^2 \les \frac{1}{ t^\eta }
            \left( \fint_Q \abs{u}^2 \right)^\frac{ \eta }{2}
            \left( \fint_Q \abs{ T u }^2 \right)^{ 1- \eta /2 }
        \end{equation}
        for all $t \in (0,1]$, $Q \in \Box_t$ and $u \in \IW^{1,2}(\bO)$.
    \end{corollary}
    \begin{proof}
        We use Jensen's inequality, Lemma~\ref{lem:interp ineq}, $ t \leq 1 $ and $ a^{ 1- \eta /2 } + b^{ 1- \eta /2 } \leq 2 (a+b)^{ 1- \eta /2 } $ for $ a,b >0 $, to estimate
        \begin{align}
            \abs{ \fint_Q Tu }^2
            &= \abs{ \fint_Q u \, }^2
            + \abs{ \fint_Q \nabla u }^2  \\
            &\les \fint_Q \abs{u}^2
            + \dfrac{1}{t^\eta} \left( \fint_Q\abs{u}^2\right)^{\frac{\eta}{2}}
            \left(\fint_Q\abs{\nabla u}^2\right)^{1-\frac{\eta}{2}} \\
            &\les \frac{1}{ t^\eta }
            \left( \fint_Q \abs{u}^2 \right)^{ \frac{ \eta }{2} }
            \left( \left( \fint_Q \abs{u}^2 \right)^{ 1- \frac{ \eta }{2} }
            + \left( \fint_Q \abs{ \nabla u }^2 \right)^{ 1- \frac{ \eta }{2} }
            \right) \\
            &\les \frac{1}{ t^\eta }
            \left( \fint_Q \abs{u}^2 \right)^{ \frac{ \eta }{2} }
            \left( \fint_Q \abs{Tu}^2 \right)^{ 1- \frac{ \eta }{2} } \, .
            \qedhere
        \end{align}
    \end{proof}

    We can now control the first term on the right-hand side of~\eqref{eq:remove smoothing}.
    \begin{proposition}
        \label{prop:est for A_t (1-P_t)}
        There holds the square function estimate
        \begin{equation*}
            \int_0^1 \norm{ \bgamma_t \mathcal{A}_t
            ( 1- \bP_t ) T_\bAr{D} u }^2_2 \ddt
            \les \norm{ T_\bAr{D} u }^2_2 \, ,
            \qquad  u \in \IW^{1,2}_\bAr{D} ( \bO ) \, .
        \end{equation*}
    \end{proposition}

    \begin{proof}
        Recall from Proposition~\ref{prop:contraction} that the averaging operator $ \mathcal{A}_t $ is a projection, that is, $ \mathcal{A}_t^2 = \mathcal{A}_t $.
        By the uniform boundedness of $ \bgamma_t \mathcal{A}_t $ from Lemma~\ref{lemma:13.6 in lectures}, we have
        \begin{equation*}
            \norm{ \bgamma_t \mathcal{A}_t (1- \bP_t ) T_\bAr{D} u }_2
            = \norm{ ( \bgamma_t \mathcal{A}_t )
            \mathcal{A}_t (1- \bP_t ) T_\bAr{D} u }_2
            \les \norm{ \mathcal{A}_t (1- \bP_t ) T_\bAr{D} u }_2 \, .
        \end{equation*}
        So, we are left to prove
        \begin{align}
            \int_0^1
            \norm{ \mathcal{A}_t (1-\bP_t) T_\bAr{D} u }^2_2 \ddt
            \les \norm{ T_\bAr{D} u }^2_2 \, .
        \end{align}
        To establish this estimate, let us assume for the moment that we have the bound
        \begin{equation}
            \label{eq:needed for schur}
            \norm{ \mathcal{A}_s (1- \bP_s ) \bQ_t^* v }_2
            \les \zeta (s/t) \norm{v}_2 \, ,
            \quad v = T^*_\bAr{D}  F  \text{ for some }
            F \in \dom ( M_\bAr{D} ) ,
        \end{equation}
        for all $ s \in (0,1] $, $ t>0 $, and with some non-negative function $ \zeta \in \L^1 ( (0, \infty ), \d \tau / \tau ) $.
        Since $ \mathcal{A}_s (1- \bP_s ) $ is a bounded operator on $ \L^2 ( \bO )^N $, the Calderón reproducing formula in Lemma~\ref{lemma:qe for Qt} with $ F = T_\bAr{D} u $ gives
        \begin{equation*}
            \mathcal{A}_s (1- \bP_s ) T_\bAr{D} u
            = 2 \int_0^\infty \mathcal{A}_s (1- \bP_s )
            \bQ_t^* \bQ_t T_\bAr{D} u \ddt \, .
        \end{equation*}
        Hence, as $ v \coloneqq \bQ_t T_\bAr{D} u = t T^*_\bAr{D}  (1+ t^2 M_\bAr{D} )^{-1} T_\bAr{D} u $ is admissible for~\eqref{eq:needed for schur}, Young's convolution inequality with indices $ 1+ \frac{1}{2} = \frac{1}{1} + \frac{1}{2} $ on the multiplicative group $ (0, \infty ) $ with measure $ \ddt $ gives
        \begin{align*}
            \int_0^1 \norm{ \mathcal{A}_s
            (1- \bP_s ) T_\bAr{D} u }^2_2 \,
            \frac{ \d s }{s}
            &\les \int_0^1 \left(
            \int_0^\infty \norm{ \mathcal{A}_s (1- \bP_s ) \bQ_t^* \bQ_t
            T_\bAr{D} u }_2 \ddt \right)^2 \, \frac{ \d s }{s} \\
            &\les \int_0^\infty \left( \int_0^\infty
            \zeta (s/t) \norm{ \bQ_t T_\bAr{D} u }_2 \ddt
            \right)^2 \, \frac{ \d s }{s} \\
            &\les \left( \int_0^\infty \zeta ( \tau )
            \, \frac{ \d \tau }{ \tau } \right)^2
            \int_0^\infty \norm{ \bQ_t T_\bAr{D} u }^2_2
            \ddt
            \les \norm{ T_\bAr{D} u }^2_2 \, ,
        \end{align*}
        where the last step above follows from the quadratic estimate in Lemma~\ref{lemma:properties of smoothing operators}.

        We turn our attention to proving~\eqref{eq:needed for schur}.
        By Definition~\ref{def:smooting operators}, one readily checks the relations
        \begin{equation}
            \label{eq:commuting t and s}
            \bP_s \bQ_t^* = \frac{t}{s} \bP_t \bQ_s^* \ ,
            \quad (1-\bP_s) \bQ_t^* = \frac{s}{t} (1-\bP_t) \bQ_s^* \, .
        \end{equation}
        Indeed, this is proved by commuting resolvents of $ M_\bAr{D} $. For instance, to see the second identity, calculate on the dense subset $ \IW^{1,2}_\bAr{D} ( \bO ) $ of $ \L^2 ( \bO )^m $ that
        \begin{align*}
            (1- \bP_s ) \bQ_t^*
            &= s^2 M_\bAr{D} (1+ s^2 M_\bAr{D} )^{-1} t (1+ t^2 M_\bAr{D} )^{-1} T_\bAr{D} \\
            &= \frac{s}{t} \, t^2 M_\bAr{D} (1+ t^2 M_\bAr{D} )^{-1}
            s (1+ s^2 M_\bAr{D} )^{-1} T_\bAr{D} \\
            &= \frac{s}{t} (1- \bP_t ) \bQ_s^* \, .
        \end{align*}
        This identity extends to $ \L^2 ( \bO )^m $ by continuity.
        Next, let $ F \in \dom ( M_\bAr{D} ) $ and $ v = T^*_\bAr{D}  F $.
        For $ s \leq t $, uniform $ \L^2 $-bounds (see Proposition~\ref{prop:contraction} and Lemma~\ref{lemma:properties of smoothing operators}) give
        \begin{equation*}
            \norm{ \mathcal{A}_s (1-\bP_s) \bQ_t^* v }_2
            = \frac{s}{t} \norm{ \mathcal{A}_s (1-\bP_t) \bQ_s^* v }_2
            \les \frac{s}{t} \norm{v}_2 \, .
        \end{equation*}
        For $ t \leq s $, we have
        \begin{equation*}
            \norm{ \mathcal{A}_s (1-\bP_s) \bQ_t^* v }_2
            \les \norm{ \mathcal{A}_s \bQ_t^* v }_2 + \norm{ \bP_s \bQ_t^* v }_2
            \les \norm{ \mathcal{A}_s \bQ_t^* v }_2 + \frac{t}{s} \norm{v}_2 \, .
        \end{equation*}
        So, we are left with estimating the first term on the right-hand side.
        Since $ F \in \dom ( M_\bAr{D} ), $ we have $ v = T^*_\bAr{D}  F \in \dom(T_\bAr{D}) $, so $ \bQ_t^* v = t (1+ t^2 M_\bAr{D} )^{-1} T_\bAr{D} v $.
        Writing out the averaging operator at scale $ s $, we have
        \begin{align}
        \label{eq:decomp_schur}
            \norm{ \mathcal{A}_s \bQ_t^* v }^2_2
            = \sum_{ R \in \Box_s } \abs{R}
            \left| \fint_R t \bP_t T_\bAr{D} T^*_\bAr{D}  F
            \right|^2 \, .
        \end{align}
        Combining $ T_\bAr{D} $ and $ T^*_\bAr{D} $ to $M_\bAr{D}$, we can \enquote{commute} $ \bP_t M_\bAr{D} \subseteq M_\bAr{D} \bP_t $, which yields
        \begin{align}
            \left| \fint_R t \bP_t T_\bAr{D} T^*_\bAr{D}  F
            \right|^2 = \left| \fint_R t M_\bAr{D} \bP_t F
            \right|^2 = \left| \fint_R t T_\bAr{D} T^*_\bAr{D}  \bP_t F
            \right|^2.
        \end{align}
        Using Corollary~\ref{cor:inhom interp ineq} and Hölder's inequality (in the sequence space $\ell^2$), we can thus estimate the right-hand side of~\eqref{eq:decomp_schur} by
        \begin{align*}
            &\frac{1}{ s^\eta }
            \sum_{ R \in \Box_s } \abs{R}
            \left( \fint_R \abs{ t T^*_\bAr{D}  \bP_t F }^2
            \right)^{ \eta /2}
            \left( \fint_R \abs{ t T_\bAr{D} T^*_\bAr{D} \bP_t F }^2
            \right)^{ 1- \eta /2 } \\
            \leq &\frac{ t^\eta }{ s^\eta }
            \left( \sum_{ R \in \Box_s }
            \int_R \abs{ T^*_\bAr{D}  \bP_t F }^2
            \right)^{ \eta /2 }
            \left( \sum_{ R \in \Box_s }
            \int_R \abs{ tT_\bAr{D} T^*_\bAr{D} \bP_t F }^2
            \right)^{ 1- \eta /2 } \\
            = &\frac{ t^\eta }{ s^\eta }
            \norm{ T^*_\bAr{D}  \bP_t F }_2^\eta
            \norm{ t M_\bAr{D} \bP_t F }_2^{ 2- \eta }
            \, .
        \end{align*}
        Now, since $ t M_\bAr{D} \bP_t F = t \bP_t M_\bAr{D} F = \bQ_t^* T^*_\bAr{D} F = \bQ_t^* v $, 
        Lemma~\ref{lemma:properties of smoothing operators} gives
        $\| t M_\bAr{D} \bP_t F \|_2 \lesssim \| v \|_2$.
        Similarly, the square root property for $ M_\bAr{D} $ applied twice, the fact that $ F \in \dom ( M_\bAr{D} ) \subseteq \dom ( \sqrt{M_\bAr{D}} ) $ and uniform boundedness of $ \{ \bP_t \}_{ t>0 } $ give
        \begin{equation*}
            \norm{ T^*_\bAr{D}  \bP_t F }_2 = \normm{ \sqrt{M_\bAr{D}} \bP_t F }_2
            = \normm{ \bP_t \sqrt{M_\bAr{D}} F }_2
            \les \normm{ \sqrt{M_\bAr{D}} F }_2 = \norm{ T^*_\bAr{D}  F }_2 = \norm{v}_2 \, .
        \end{equation*}
        Putting things together, we arrive at~\eqref{eq:needed for schur} with $ \zeta ( \tau ) = \min \{ \tau , \tau^{-1} + \tau^{-\eta} \} $, which belongs to $ \L^1 ( (0, \infty ) , \d \tau / \tau ) $.
    \end{proof}

    \subsection{Full principal part approximation}
    \label{sec:full principal part approx}

    Let us summarize the different estimates of the previous sections in a single result for future reference.

    \begin{proposition}
        \label{prop:full principal part approx}
        There holds the quadratic estimate for the principal part approximation
        \begin{equation*}
            \int_0^1
            \norm{ ( \bTheta_t - \bgamma_t \mathcal{A}_t )
            T_\bAr{D} u }^2_2 \ddt
            \les \norm{ T_\bAr{D} u }_2^2 \, ,
            \qquad u \in \IW^{1,2}_\bAr{D} ( \bO ) \, .
        \end{equation*}
    \end{proposition}
    \begin{proof}
        Write
        \begin{equation*}
            \bTheta_t - \bgamma_t \mathcal{A}_t
            = ( \bTheta_t - \bgamma_t \mathcal{A}_t ) \bP_t
            + \bTheta_t ( 1- \bP_t ) - \bgamma_t \mathcal{A}_t
            ( 1- \bP_t ) \, .
        \end{equation*}
        Each term has been treated separately in Propositions~\ref{prop:smooth principal part},~\ref{prop:QE for Theta (1-P) nabla u} and~\ref{prop:est for A_t (1-P_t)}, respectively.
    \end{proof}

	\section{Carleson measure estimate and a first conclusion}
	\label{sec:carleson}

    In Section~\ref{sec:principal part approximation}, we have reduced the proof of the quadratic estimate~\eqref{eq:QE} (and hence of the Kato square root conjecture on the interior thick set $\bO$) to proving that the measure $ \norm{ \gamma_t (x) }^2 \frac{ \d x \d t }{t} $ is a Carleson measure on $ (0,1] \times \bO $, that is,
    \begin{equation}
        \label{eq:def carleson measure}
        \iint_{R(Q)} \norm{ \gamma_t (x) }^2
        \frac{ \d x \d t }{t} \les \abs{Q}
        \, , \qquad Q \in \Box \, .
    \end{equation}
    Here, $R(Q)$ %
    denotes the Carleson box over $Q$ and is given by $ R(Q) = (0, \ell (Q) ] \times Q \subseteq (0,1] \times \bO $.
    The estimate~\eqref{eq:def carleson measure} follows from the subsequent lemma, whose proof amounts to covering the unit sphere in $ \C^{ m \times N } $ by finitely many sectors.

    \begin{lemma}
        \label{lem:key estimate}
        Suppose that there is an $ \eps >0 $ such that, for each \enquote{sector}
        \begin{equation*}
            \Gamma^\eps_\nu
            = \left\{ \mu \in \C^{ m \times N }
            \setminus \{ 0 \} :
            \norm{ \frac{ \mu }{ \norm{ \mu } } - \nu } \leq \eps
            \right\} \, ,
            \quad \nu \in \C^{ m \times N } \, ,
            \ \norm{ \nu } = 1 \, ,
        \end{equation*}
        in $ \C^{ m \times N } $, there holds the estimate
        \begin{equation}
            \label{eq:key estimate}
            \iint_{ R(Q) } \norm{ \gamma_t (x)
            \one_{ \Gamma^\eps_\nu } ( \gamma_t (x) ) }^2
            \, \frac{ \d x \d t }{t}
            \les \abs{Q} \, ,
            \qquad Q \in \Box \, .
        \end{equation}
        Then $ \norm{ \gamma_t (x) }^2 \frac{ \d x \d t }{t} $ is a Carleson measure.
    \end{lemma}

    It remains to establish the bound~\eqref{eq:key estimate}.

	\subsection{Construction of Tb--type test functions}
	\label{sec:Tb}
    
    The proof of the key estimate~\eqref{eq:key estimate} relies on a $ T(b) $--argument and the construction of test functions adapted to a \enquote{cube} $ Q \in \Box $ and a vector direction $ \xi = \xi_\nu \in \C^N $. This vector direction $\xi$ is related to a matrix direction $ \nu \in \C^{ m \times N } $ appearing in the definition of the \enquote{sector} $\Gamma^\eps_\nu$.
    Indeed, later on in Proposition~\ref{prop:stopping time}, we will see that the vector direction $\xi$ provides sufficient information to obtain an estimate for the matrix direction $ \nu $.

    \begin{proposition}
        \label{prop:T(b) functions}
        There is a constant $ \eps_0 \in (0,1) $ such that, for all $ 0 < \eps < \eps_0 $, all unit vectors $ \xi \in \C^N $ and each dyadic \enquote{cube} $ Q \in \Box $, we can find a \enquote{$ \,T(b) $--type test function} $ b = b^{ \xi , Q , \eps } \in \L^2 ( \bO )^N $ with the following properties:
        \begin{enumerate}[(a)]
            \item
            $ \normm{ b^{ \xi , Q , \eps } }^2 \les \abs{ Q } $,
            \item
            $ \Re \left( \xi \cdot \fint_Q b^{ \xi , Q , \eps } \right) \geq 1 $,
            \item
            $ \iint_{ R(Q) } \abss{ \bgamma_t (x) \mathcal{A}_t b^{ \xi , Q , \eps } (x) }^2 \, \frac{ \d x \d t }{t} \les \abs{Q} / \eps^2 $.
        \end{enumerate}
    \end{proposition}

    \begin{proof}
        Fix a unit vector $ \xi \in \C^N $ and $ Q \in \Box $, and denote by $ \ell = \ell (Q) $ the \enquote{sidelength} of $ Q $. For $\eps > 0$ subject to further constraints below, define
        \begin{align}
        \label{eq:def_b}
            b \coloneqq b^{ \xi , Q , \eps }
            \coloneqq 2( 1 + \eps^2 \ell^2 M_\bAr{D} \bB )^{-1}
            ( \one_{2Q} \bar{\xi} ) \, ,
        \end{align}
        where the set $2Q$ was defined in Section~\ref{sec:dyadic}.
        Recall also the bound $|2Q| \lesssim |Q|$ from that section.
        Owing to Lemma~\ref{lem:M perturbation}, $ M_\bAr{D} \bB $ is sectorial, and hence $ b $ is well-defined.
        We are going to verify that, upon choosing $ \eps \in (0,1) $ sufficiently small, $ b $ satisfies the three properties listed above.
        Write $ \tau = \eps \ell $ for brevity.

        \textbf{Proof of (a)}:
        Since $ M_\bAr{D} \bB $ is sectorial, the resolvents $ ( 1+ \tau^2 M_\bAr{D} \bB )^{-1} $ are uniformly bounded in $ \mathcal{L} ( \L^2 ( \bO )^N ) $, so we have
        \begin{align}
        \| b \|_2 = 2 \| (1 + \tau^2 M_\bAr{D} \bB)^{-1} \one_{2Q} \overline{\xi} \|_2 \les \| \one_{2Q} \overline{\xi} \|_2 \les \abs{Q}^\frac{1}{2}.
        \end{align}
        \textbf{Proof of (b)}
        The underlying heuristic of (b) is that the average of $ b $ over $ Q $ points roughly in the direction $ \overline{\xi} $. To exploit this, we estimate the difference
        \begin{equation}
            \label{eq:diff b-xi}
            \tfrac{1}{2}b - \one_{2Q} \bar{\xi}
            = -\tau^2 M_\bAr{D} \bB ( 1+ \tau^2 M_\bAr{D} \bB )^{-1}
            \one_{2Q} \bar{\xi}
            = \tau^2 T_\bAr{D} v,
        \end{equation}
        where $ v \coloneqq -T^*_\bAr{D} \bB ( 1+ \tau^2 M_\bAr{D} \bB )^{-1} \one_{2Q} \bar{\xi} $, with the hope that this difference becomes small in the end. Since this difference is a \enquote{gradient field}, Corollary~\ref{cor:inhom interp ineq} yields
        \begin{align}
            \abs{ \fint_Q T_\bAr{D} v }^2 &\les \frac{1}{\ell^\eta} \left( \fint_Q \abs{v} ^2 \right)^{\frac{\eta}{2}} \left( \fint_Q \abs{ T_\bAr{D} v } ^2 \right)^{1-\frac{\eta}{2}} \\
            &\leq \frac{1}{\ell^\eta} \frac{1}{\abs{Q}} \|v\|_2^\eta \| T_\bAr{D} v \|_2^{2-\eta} \, .
        \end{align}
        By~\eqref{eq:diff b-xi} and~(a), we have
        \begin{equation*}
            \norm{ T_\bAr{D} v }_2 = \tau^{-2} \norm{ \tfrac{1}{2} b - \one_{2Q} \bar{\xi} }_2
            \les \tau^{-2} \abs{Q}^{1/2} \, .
        \end{equation*}
        Similarly, by the uniform boundedness from Lemma~\ref{lem:M perturbation} and the doubling property for dyadic \enquote{cubes}, we can control
        \begin{equation*}
            \|v\|_2 = \| T^*_\bAr{D}  \bB ( 1+ \tau^2 M_\bAr{D} \bB )^{-1} \one_{2Q} \bar{ \xi } \|_2  \les \tau^{-1} \| \one_{2Q} \bar{ \xi } \|_2 \approx \tau^{-1} \abs{Q}^\frac{1}{2} .
        \end{equation*}
        Collecting the powers, we arrive at
        \begin{equation*}
            \left| \fint_Q \tfrac{1}{2} b - \one_{2Q} \bar{\xi}
            \right|^2
            = \tau^4 \left| \fint_Q
            T_\bAr{D} v \right|^2
            \les \frac{ \tau^{4- \eta -4 +2 \eta } }
            { \ell^\eta }
            = \eps^\eta \, .
        \end{equation*}
        Write $C^2$ for the implicit constant in the preceding bound.
        Now, assertion (b) follows from
        \begin{equation*}
            \Re \left( \xi \cdot \fint_Q b \right) = 2 + 2 \Re \left( \xi \cdot \fint_Q \tfrac{1}{2} b - \one_{2Q} \bar{ \xi }  \right) \geq 2 - 2C \eps^\frac{\eta}{2},
        \end{equation*}
        if we choose $\eps$ small enough so that $2C\eps^{\nicefrac{\eta}{2}} \leq 1$.

        \textbf{Proof of (c)}
        Our choice of $ T(b) $--type test function $ b $ is not in the range of $ T_\bAr{D} $, so we cannot directly apply the principal part approximation estimate from Proposition~\ref{prop:full principal part approx} to it.
        However, recalling~\eqref{eq:diff b-xi}, we split the integral as follows
        \begin{align*}
            \iint_{ R(Q) } &\abs{ \bgamma_t (x)
            (\mathcal{A}_t b)(x) }^2
            \, \frac{ \d x \d t }{t}
            \les \iint_{ R(Q) } \abs{ ( \bgamma_t
            \mathcal{A}_t - \bTheta_t ) b }^2 \,
            + \abs{ \bTheta_t b }^2
            \, \frac{ \d x \d t }{t} \\
            &\les \iint_{ R(Q) } \abs{ ( \bgamma_t
            \mathcal{A}_t - \bTheta_t )
            ( \tfrac{1}{2} b - \one_{2Q} \bar{ \xi } ) }^2
            + \abs{ ( \bgamma_t
            \mathcal{A}_t - \bTheta_t ) \one_{2Q} \bar{ \xi } }^2
            + \abs{ \bTheta_t b }^2
            \, \frac{ \d x \d t }{t} \\
            &\eqqcolon I + II + III \, .
        \end{align*}
        Since $ \tfrac{1}{2} b- \one_{2Q} \bar{ \xi } $ is in the range of $ T_\bAr{D} $, Proposition~\ref{prop:full principal part approx} and property~(a) above control the first term as
        \begin{equation*}
            I \les \int_0^1
            \normm{ ( \bgamma_t \mathcal{A}_t
            - \bTheta_t ) (\tfrac{1}{2} b- \one_{2Q} \bar{ \xi } ) }^2_2 \,
            \ddt
            \les \norm{b}^2_2 + \normm{ \one_{2Q} \bar{ \xi } }^2_2
            \les \abs{Q} .
        \end{equation*}
        Since we have $|Q| \lesssim |Q|/\eps^2$, this completes the treatment of the first term.
        The second term can be estimated via off-diagonal estimates and the definition of the principal part approximation, similarly to the proof of Lemma~\ref{lemma:13.6 in lectures}.
        Recall that we seek to control
        \begin{equation*}
            \int_0^\ell
            \normm{ \one_Q ( \bgamma_t \mathcal{A}_t
            - \bTheta_t ) \one_{2Q} \bar{\xi} }^2_2 \ddt \, .
        \end{equation*}
         On the \enquote{cube} $ Q $, $ \one_{2Q} \overline{\xi}$ is constant, so we have $ \mathcal{A}_t \one_{2Q} \bar{\xi} (x) = \bar{\xi} $ for almost every $ x \in Q $. This yields $ \bgamma_t \mathcal{A}_t \one_{2Q} \bar{\xi} = \bTheta_t \bar{\xi} $ on $ Q $.
        With the notation from the proof of Lemma~\ref{lemma:13.6 in lectures}, this leads to
        \begin{align}\
        \label{eq:Tb_ode_term}
            \one_Q ( \bgamma_t \mathcal{A}_t
            - \bTheta_t ) \one_{2Q} \bar{\xi}
            = \one_Q \bTheta_t ( \bar{\xi} - \one_{2Q} \bar{\xi} )
            = \sum_{ j \geq 1 } \one_Q
            \bTheta_t ( \one_{ C_j(Q) \setminus 2Q } \bar{\xi} ) \, ,
        \end{align}
        where we split $ \bO $ into truncated annuli $ C_j(Q)$.
        Observe that, for $j\geq 1$ and $t \leq 1$, we have $$\e^{-c \frac{2^j}{t}} = \e^{-c \frac{2^{j-1}}{t}} \e^{-c \frac{2^{j-1}}{t}} \leq \e^{-c 2^{j-1}} \e^{-\frac{c}{t}}.$$
        Then, $ \L^2 $ off-diagonal estimates (see Proposition~\ref{prop:ode}) control the $ \L^2 ( \bO ) $-norm of~\eqref{eq:Tb_ode_term} by
        \begin{align*}
            \sum_{ j \geq 1 }
            \e^{ -c \frac{ \dist( C_j(Q) \setminus 2Q , Q ) }{t} }
            \normm{ \one_{ C_j(Q) \setminus 2Q } \bar{\xi} }_2
            &\les \sum_{ j \geq 1 }
            \e^{ -c \frac{2^j}{t} } \, 2^{\frac{jn}{2}} \abs{Q}^{1/2} \\
            &\leq \Big( \sum_{ j \geq 1 }
            \e^{ -c 2^{j-1} } \, 2^{\frac{jn}{2}} \Big)
            \e^{ -c/t }\abs{Q}^{1/2} \\
            &\les t \abs{Q}^{1/2} \, ,
        \end{align*}

        where the last step follows from convergence of the series and estimating $ \e^{-s} \les \frac{1}{s} $ for $ s \geq 0 $. It follows that
        \begin{equation*}
            \int_0^\ell
            \normm{ \one_Q ( \bgamma_t \mathcal{A}_t
            - \bTheta_t ) \one_{2Q} \bar{\xi} }^2_2 \ddt
            \les \int_0^1 t^2 \abs{Q} \ddt
            \les \abs{Q} \, .
        \end{equation*}

        For the last term $III$, observe that resolvents of $ M_\bAr{D} \bB $ map into $ \dom ( M_\bAr{D} \bB ) \subseteq \dom( T^*_\bAr{D} \bB ) $, so we have
        \begin{align}
            \bTheta_t b &= 2t ( 1+t^2 \bcL )^{-1}
            T^*_\bAr{D} \bB ( 1+ \tau^2 M_\bAr{D} \bB )^{-1} \one_{2Q} \bar{ \xi } \\
            &= 2t ( 1+t^2 \bL )^{-1}
            T^*_\bAr{D} \bB ( 1+ \tau^2 M_\bAr{D} \bB )^{-1} \one_{2Q} \bar{ \xi }
            \, \\
            &= t(1 + t^2 \bL)^{-1} T_\bAr{D}^* \bB b.
        \end{align}
        By Lemma~\ref{lem:M perturbation}, we have $ \norm{ T^*_\bAr{D} \bB b }_2 \les \tau^{-1} \normm{ \one_{2Q} \bar{ \xi } }_2 \les \tau^{-1} \abs{Q}^{1/2} $.
        Hence, uniform boundedness of the resolvents $ ( 1+ t^2 \bL )^{-1} $ in $ \mathcal{L} ( \L^2 ( \bO )^m ) $ gives
        \begin{equation*}
            III \les \eps^{-2} \ell^{-2} \int_0^\ell t^2
            |Q| \ddt
            = %
            \frac{ \abs{Q} }{ \eps^2 } \, .
            \qedhere
        \end{equation*}

    \end{proof}

	\subsection{Stopping time argument}
	\label{sec:stopping}

    We can now exploit the dyadic structure on $ \bO $ and the properties of the $ T(b) $--test functions constructed in Proposition~\ref{prop:T(b) functions} to derive the key estimate~\eqref{eq:key estimate}.
    For this, fix a matrix $ \nu \in \mathcal{L} ( \C^N , \C^m ) $ with $ \norm{ \nu } =1 $. As the adjoint matrix $ \nu^* \in \mathcal{L} ( \C^m , \C^N ) $ also has norm $ 1 $, we find unit vectors $ \xi = \xi_\nu \in \C^N $ and $ \eta = \eta_\nu \in \C^m $ with $ \overline{\xi} = \nu^* \eta $.
    To prove~\eqref{eq:key estimate}, the main point is to deduce from property~(b) in Proposition~\ref{prop:T(b) functions} that $  \mathcal{A}_t b^{ \xi , Q , \eps } (x) $ points roughly in direction $ \overline{\xi} $, so that $ \mu \mathcal{A}_t b^{ \xi , Q , \eps } (x) $ is not too small whenever $ \mu \in \Gamma^\eps_\nu $, at least for all $ x $ in a \enquote{substantial subset} of $ Q $.

    \begin{proposition}
        \label{prop:stopping time}
        There is some $ \eps \in (0, \eps_0 ) $ and some $ \kappa >0 $ such that, for all matrices $ \nu \in \mathcal{L}(\C^m, \C^N) $ of unit norm and all $ Q \in \Box $, there exist pairwise disjoint dyadic children $ Q_j \subseteq Q $ such that for the sets
        \begin{equation}
            \label{eq:def E E*}
            E(Q) = Q \setminus \bigcup_j Q_j
            \subseteq Q \ , \qquad
            E^* (Q) = R(Q) \setminus \bigcup_j
            R(Q_j) \, ,
        \end{equation}
        there hold the following properties:
        \begin{enumerate}[(a)]
            \item
            $ \abs{ E(Q) } \geq \kappa \abs{Q} $,
            \item
            $ \abs{ \mu (\mathcal{A}_t b^{ \xi , Q , \eps })(x) } \geq \tfrac{1}{2} \norm{ \mu } $ for all $ (x,t) \in E^* (Q) $ and $ \mu \in \Gamma^\eps_\nu $.
        \end{enumerate}
    \end{proposition}
    \begin{proof}
        Having Proposition~\ref{prop:T(b) functions} at hand, we can follow the Euclidean proof (for instance in~\cite[Lem.~14.8]{ISEM}) almost verbatim. Hence, we are only going to present the necessary changes.

        The selection of the dyadic children $Q_j \subseteq Q$ and the proof of part~(a) is completely analogous to the Euclidean case.
        Hence, we concentrate on~(b). Let $(x,t) \in E^*(Q)$ and $\mu \in \Gamma^\eps_\nu$.
        For brevity, set $v \coloneqq (\mathcal{A}_t b^{ \xi , Q , \eps })(x) \in \C^N$.
        As in the Euclidean case, the construction of $E^*(Q)$ implies the bounds $|v| \leq \tfrac{1}{4\eps}$ and $\Re (\xi \cdot v) \geq \tfrac{3}{4}$.
        Using the triangle inequality, the definition of $\Gamma^\eps_\nu$ and the upper bound for $|v|$, deduce
        \begin{align}
        \label{eq:stopping1}
            \abs{\frac{\mu}{\norm{\mu}} v} \geq |\nu v| - \norm{\frac{\mu}{\norm{\mu}} - \nu} |v| \geq |\nu v| - \eps |v| \geq |\nu v| - \frac{1}{4}.
        \end{align}
        Next, by the Cauchy--Schwarz inequality in $ \C^m $ and the relation $ \overline{\xi} = \nu^* \eta $, we obtain the lower bound
        \begin{align}
            \abs{ \nu v }
            \geq \Re \langle \eta, \nu v \rangle
            = \Re \langle \nu^* \eta, v \rangle
            = \Re \langle \overline{\xi}, v  \rangle = \Re (\xi \cdot v) \, ,
        \end{align}
        where the brackets denote the inner products in either $\C^m$ or $\C^N$.
        Plugging this bound back into~\eqref{eq:stopping1} and using the lower bound for $\Re (\xi \cdot v)$, deduce
        \begin{align}
            \abs{\frac{\mu}{\norm{\mu}} v} \geq \Re (\xi \cdot v) - \frac{1}{4} \geq \frac{1}{2}.
        \end{align}
        Multiplying this bound by $\norm{\mu}$ yields the claim.
    \end{proof}

    To conclude the proof of the key estimate~\eqref{eq:key estimate}, we cite the John--Nirenberg lemma for Carleson measures. %
    The result is well-known and a detailed proof using the dyadic \enquote{cubes} from Definition~\ref{eq:dyadic_cubes} is presented in~\cite[p.~191]{Diss}.

    \begin{lemma}[John--Nirenberg lemma for Carleson measures]
        \label{lemma:john-nirenberg}
        Let $ \sigma $ be a Borel measure on $ \bO \times (0,1] $ such that there are $ K , \kappa >0 $ with the following property:
        Each dyadic \enquote{cube} $ Q \in \Box $ has pairwise disjoint dyadic children $ Q_j \in \Box $ such that the sets $ E(Q) \subseteq Q $ and $ E^* (Q) $ defined in~\eqref{eq:def E E*} satisfy
        \begin{enumerate}[(a)]
            \item
            $ \abs{ E(Q) } \geq \kappa \abs{Q} $.
            \item
            $ \sigma ( E^* (Q) ) \leq K \abs{Q} $.
        \end{enumerate}
        Then $ \sigma $ is a Carleson measure with norm at most $ K / \kappa $.
    \end{lemma}

    \begin{proof}[Proof of the key estimate~\eqref{eq:key estimate}]
        Fix a unit matrix $ \nu \in \mathcal{L}(\C^m, \C^N) $ and set $ \xi = \xi_\nu \in \C^N $.
        By Proposition~\ref{prop:stopping time}~(b) and Proposition~\ref{prop:T(b) functions}~(c), we compute
        \begin{align*}
            \iint_{ E^* (Q) }
            \norm{ \bgamma_t (x) \one_{ \Gamma^\eps_\nu }
            ( \bgamma_t (x) ) }^2 \, \frac{ \d x \d t }{t}
            &\leq 4 \iint_{ E^* (Q) }
            \abs{ \bgamma_t (x) (\mathcal{A}_t
            b^{ \xi , Q , \eps })(x)
            \one_{ \Gamma^\eps_\nu } ( \bgamma_t (x) ) }^2
            \, \frac{ \d x \d t }{t} \\
            &\leq 4 \iint_{ R(Q) }
            \abs{ \bgamma_t (x) (\mathcal{A}_t
            b^{ \xi , Q , \eps })(x) }^2
            \, \frac{ \d x \d t }{t}
            \les \frac{ \abs{Q} }{ \eps^2 } \, .
        \end{align*}
        By Lemma~\ref{lemma:john-nirenberg}, this and Proposition~\ref{prop:stopping time}~(a) ensure  that $ \d \sigma = \normm{ \bgamma_t (x) \one_{ \Gamma^\eps_\nu } ( \bgamma_t (x) ) }^2 \tfrac{ \d x \d t }{t} $ is a Carleson measure, which yields the asserted estimate~\eqref{eq:key estimate}.
    \end{proof}

    We are now done with the proof of the quadratic estimate~\eqref{eq:QE}, and hence of Theorem~\ref{thm:main}, under the additional assumption that $\bO$ is interior thick.
    To convince the reader, we give below a short proof summarizing the different reduction steps of our proof.

    \begin{proof}[Proof of Theorem~\ref{thm:main} when $\bO$ is interior thick]
        In Proposition~\ref{prop:reduc to QE}, we have seen that the Kato square root property is equivalent to~\eqref{eq:QE}.
        The integral on the left-hand side can be split as
        \begin{align*}
            \int_0^\infty \normm{ \bTheta_t T_\bAr{D} u }^2_2 \ddt
            \les
            &\int_1^\infty \normm{ \bTheta_t (1- \bP_t ) T_\bAr{D} u }^2_2 \ddt
            + \int_1^\infty \normm{ \bTheta_t \bP_t T_\bAr{D} u }^2_2 \ddt \\
            + &\int_0^1
            \normm{ ( \bTheta_t - \bgamma_t \mathcal{A}_t )
            T_\bAr{D} u }^2_2 \ddt
            + \int_0^1
            \normm{ \bgamma_t \mathcal{A}_t T_\bAr{D} u }^2_2 \ddt \, .
        \end{align*}
        The first and second terms are controlled by Proposition~\ref{prop:QE for Theta (1-P) nabla u} and Proposition~\ref{prop:reduction to finite time}, respectively, and the third one is estimated in Proposition~\ref{prop:full principal part approx}.
        The last integral is bounded by
        \begin{equation*}
            \int_0^1 \int_\bO \abs{ (\mathcal{A}_t T_\bAr{D} u)(x) }^2
            \norm{ \bgamma_t (x) }^2 \, \frac{ \d x \d t }{t} \, .
        \end{equation*}
        By Proposition~\ref{prop:dyadic Carleson}, it then suffices to prove that $ \d \sigma = \norm{ \bgamma_t (x) }^2 \tfrac{ \d x \d t }{t} $ is a Carleson measure on $ \bO \times (0,1] $, that is, the validity of~\eqref{eq:def carleson measure}.
        Splitting over angular sectors $ \Gamma^\eps_\nu $ in $ \mathcal{L}(\C^m, \C^N) $ reduces this to finding an $ \eps >0 $ such that~\eqref{eq:key estimate} holds true for all (matrix) directions $ \nu \in \mathcal{L}(\C^m, \C^N) $. The latter was accomplished right above via the usage of $ T(b) $--type test functions and a stopping time argument.
    \end{proof}

	\section{Removing the thickness condition}
	\label{sec:no_thickness}

    In this final section, we remove the interior thickness condition imposed in the approach performed throughout Sections~\ref{sec:reduction}-\ref{sec:carleson}, thereby completing the proof of Theorem~\ref{thm:main} in the general case. To do so, we employ the transference strategy pioneered in~\cite[Sec.~6]{BEH-Kato}. If the Dirichlet parts were the same for all components, their argument would apply out of the box. For the reader's convenience, and since we want to extend their result to varying Dirichlet parts for different components of the system, we recap their argument and comment on the necessary changes.

    Let $\bO \subseteq \R^n$ be open and $\bD_j \subseteq \bd\bO$ be closed, and assume that $\bO$ is locally uniform near $\partial \bO \setminus \bD_j$ for every $j=1,\dots, m $. Put $\bAr{D} \coloneqq (\bD_j)_{j=1}^m$.
    Now, fix a decomposition of $\bO$ into pairwise disjoint open sets $O_i$. Define the array $\Ar{D}_i$ via $(\Ar{D}_i)_j \coloneqq \bD_j \cap \partial O_i$
    and consider elliptic operators $L_i$ on $O_i$ subject to a vanishing trace condition on $\Ar{D}_i \subseteq (\partial O_i)^m$ and whose coefficients are the restriction of the coefficients of $\bL$ to $O_i$. Then, the following correspondence holds~\cite[Prop.~6.8]{BEH-Kato}.

    \begin{proposition}[Transference principle]
        \label{prop:transference}
        The following are equivalent:
        \begin{enumerate}
            \item[(i)] $\dom(\sqrt{\bL}) = \IW^{1,2}_\bAr{D}(\bO)^m$ with $\| \sqrt{\bL} u \|_2 \approx \| u \|_{\IW^{1,2}(\bO)}$,
            \item[(ii)] $\dom(\sqrt{L_i}) = \IW^{1,2}_{\Ar{D}_i}(O_i)$ with $\| \sqrt{L_i} u \|_2 \approx \| u \|_{\IW^{1,2}(O_i)}$ for all $i$, where the implicit constants are independent of $i$.
        \end{enumerate}
    \end{proposition}

    Now, the central idea in~\cite{BEH-Kato} is as follows: starting with a triple $(O,\Ar{D},L)$ as in Theorem~\ref{thm:main}, they construct a new triple $(\bO,\bAr{D},\bL)$ that satisfies all assumptions of Theorem~\ref{thm:main}, contains $(O,\Ar{D},L)$ as a \enquote{component} in the sense explained above, and such that $\bO$ is in addition interior thick. Then, since we have proved Theorem~\ref{thm:main} under the additional assumption that $\bO$ is interior thick, item~(i) of Proposition~\ref{prop:transference} is verified. Therefore, Proposition~\ref{prop:transference} provides the square root property for all triples quantified in~(ii), including the triple $(O,\Ar{D},L)$, which completes the proof of Theorem~\ref{thm:main} in the general case.

    Hence, it only remains to clarify the construction of the triple $(\bO,\bAr{D}, \bL)$. By assumption, $O$ is locally an $(\eps,\delta)$-domain near $\partial O \setminus D_j \eqqcolon N_j$ for all $j=1,\dots,m$. Let $\Sigma$ denote a grid of open axis-parallel cubes of diameter $\nicefrac{\delta}{8}$ in $\R^n$. Let $\Sigma'$ contain those cubes $Q$ from $\Sigma$ for which $\overline{Q}$ intersects $D$ but which have empty intersection with the tube of size $\nicefrac{\delta}{4}$ around $N^* \coloneqq \cup_{j=1}^m N_j$. Now, put
    \begin{align}
        \bO \coloneqq O \cup \bigcup_{Q\in \Sigma'} (Q \setminus \partial O) \quad \text{and} \quad (\bAr{D})_j \coloneqq \partial \bO \setminus N_j.
    \end{align}

    In~\cite{BEH-Kato}, all components were subject to the same Dirichlet boundary conditions, so that $N^* = N$ in that paper. In our more general setting, all desired properties for $(\bO,\bAr{D})$ except the interior thickness condition immediately follow by the same proof as in~\cite[Sec.~6.3]{BEH-Kato}.

    To show the interior thickness condition for $\bO$, the authors of~\cite{BEH-Kato} argue first that only balls centered in $x\in O$ have to be considered. Now, we need to distinguish cases: first, if $x$ is in the $\nicefrac{\delta}{2}$-tube around $N^*$, then $x$ is in the $\nicefrac{\delta}{2}$-tube around $N_j$ for \emph{some} index $j=1,\dots,m$, and thickness in $x$ is provided by the corkscrew condition in Proposition~\ref{prop:corkscrew}. Otherwise, as in~\cite{BEH-Kato}, it follows that $x \in \overline{Q}$ for some $Q \in \Sigma'$, so that thickness is guaranteed by the respective property of the cube $Q$. Thus, $\bO$ is indeed interior thick. Finally, the operator $L$ can be extended to $\bO$ by setting it equal to $-\Delta_\bAr{D} + 1$ on $\bO \setminus O$. This completes the proof. \hfill $\square$ \\


\begin{thebibliography}{11}
        \bibitem{Achache-Ouhabaz}
        \textsc{M.~Achache} and \textsc{E.-M.~Ouhabaz}.
        \newblock {\em Lions' maximal regularity problem with {$H^{\frac12}$}-regularity in time\/}.
        \newblock J. Differential Equations \textbf{266} (2019), no.~6, 3654--3678.

        \bibitem{Adams-Hedberg}
        \textsc{D.~Adams} and \textsc{L.~Hedberg}.
        \newblock Function {S}paces and {P}otential {T}heory.
        Grundlehren der Mathematischen Wissenschaften,  vol.~314,
        \newblock Springer, Berlin, 1996.

        \bibitem{Parabolic_SecondOrder}
        \textsc{A.~Ataei}, \textsc{M.~Egert}, and \textsc{K.~Nystr{\"o}m}.
        \newblock {\em The {Kato} square root problem for weighted parabolic operators\/}.
        \newblock Anal. PDE \textbf{18} (2025), no.~1, 141--169.

        \bibitem{AAM-ArkMath}
        \textsc{P.~Auscher}, \textsc{A.~Axelsson}, and \textsc{A.~McIntosh}.
        \newblock {\em Solvability of elliptic systems with square integrable boundary data\/}.
        \newblock Ark. Mat. \textbf{48} (2010), no.~2, 253--287.

        \bibitem{AE-mixed}
        \textsc{P.~Auscher} and \textsc{M.~Egert}.
        \newblock {\em Mixed boundary value problems on cylindrical domains\/}.
        \newblock Adv. Differential Eq. \textbf{22} (2017), no.~1--2, 101--168.

        \bibitem{Block}
		\textsc{P.~Auscher} and \textsc{M.~Egert}.
		\newblock Boundary Value Problems and Hardy Spaces for Elliptic Systems with Block Structure. Progress in Mathematics, vol.~346,
		\newblock Birkhäuser, Cham, 2023.

        \bibitem{Parabolic_FirstOrder}
        \textsc{P.~Auscher}, \textsc{M.~Egert}, and \textsc{K.~Nystr{\"o}m}.
        \newblock {\em {{\(L^2\)}} well-posedness of boundary value problems for parabolic systems with measurable coefficients\/}.
        \newblock J. Eur. Math. Soc. (JEMS) \textbf{22} (2020), no.~9, 2943--3058.

        \bibitem{Kato-Square-Root-Proof}
		\textsc{P.~Auscher}, \textsc{S.~Hofmann}, \textsc{M.~Lacey}, \textsc{A.~McIntosh}, and \textsc{
			P.~Tchamitchian}.
		\newblock {\em The solution of the {K}ato square root problem for second order
			elliptic operators on {$\mathbb{R}^n$}\/}.
		\newblock Ann. of Math. (2) \textbf{156} (2002), no.~2, 633--654.

		\bibitem{AKM}
		\textsc{A.~Axelsson}, \textsc{S.~Keith}, and \textsc{A.~McIntosh}.
		\newblock {\em The {K}ato square root problem for mixed boundary value
			problems\/}.
		\newblock J. London Math. Soc. (2) \textbf{74} (2006), no.~1, 113--130.

		\bibitem{AKM-QuadraticEstimates}
		\textsc{A.~Axelsson}, \textsc{S.~Keith}, and \textsc{A.~McIntosh}.
		\newblock {\em Quadratic estimates and functional calculi of perturbed {D}irac
			operators\/}.
		\newblock Invent. Math. \textbf{163} (2006), no.~3, 455--497.

        \bibitem{Bechtel_Lp_Kato}
        \textsc{S.~Bechtel}.
        \newblock {\em {$L^p$}-estimates for the square root of elliptic systems with mixed boundary conditions~II\/}.
        \newblock J. Differential Equations \textbf{379} (2024), 104--124.

        \bibitem{seb_book}
		\textsc{S.~Bechtel}.
		\newblock Square roots of elliptic systems in locally uniform domains. Operator Theory: Advances and Applications, vol.~303,
		\newblock Birkh{\"a}user, Cham, 2024.

        \bibitem{BHT}
		\textsc{S.~Bechtel}, \textsc{R.~Brown}, \textsc{R.~Haller-Dintelmann}, and \textsc{P.~Tolksdorf}.
		\newblock{\em Sobolev extension operators for functions with partially vanishing trace}.
		\newblock To appear in Ann. Inst. Fourier (Grenoble).
		\newblock Available under: \url{https://aif.centre-mersenne.org/item/10.5802/aif.3707.pdf}.

        \bibitem{BE}
        \textsc{S.~Bechtel} and \textsc{M.~Egert}.
        \newblock {\em Interpolation theory for Sobolev functions with partially vanishing trace on irregular open sets\/}.
        \newblock J. Fourier Anal. Appl. \textbf{25} (2019), no.~5, 2733--2781.

		\bibitem{BEH-Kato}
		\textsc{S.~Bechtel}, \textsc{M.~Egert}, and \textsc{R.~Haller-Dintelmann}.
		\newblock {\em The {K}ato square root problem on locally uniform domains\/}.
		Adv. Math. \textbf{375} (2020). %

        \bibitem{BCE}
		\textsc{T.~B{\"o}hnlein}, \textsc{S.~Ciani}, and \textsc{M.~Egert}.
		\newblock {\em Gaussian estimates vs. elliptic regularity on open sets\/}.
		\newblock Math. Ann. \textbf{391} (2020), no.~2, 2709--2756.

        \bibitem{Bonifacius-Neitzel}
        \textsc{L.~Bonifacius} and \textsc{I.~Neitzel}.
        \newblock {\em Second Order Optimality Conditions for Optimal Control of Quasilinear Parabolic Equations\/}.
        \newblock{Math. Control Relat. Fields \textbf{8} (2018), no.~1, 1--34.}

		\bibitem{ChristDyadic}
		\textsc{M.~Christ}.
		\newblock {\em A {$T(b)$} theorem with remarks on analytic capacity and the
			{C}auchy integral\/}.
		\newblock Colloq. Math. \textbf{60/61} (1990), no.~2, 601--628.

		\bibitem{Diss}
		\textsc{M.~Egert}.
		\newblock{On {K}ato's conjecture and mixed boundary conditions\/}.
		\newblock PhD Thesis, Sierke Verlag, G\"ottingen, 2015. Available online \url{https://www.mathematik.tu-darmstadt.de/media/analysis/personen_analysis/egert/Thesis.pdf}.

        \bibitem{Egert_Lp_Kato}
        \textsc{M.~Egert}.
        \newblock {\em {$L^p$}-estimates for the square root of elliptic systems with mixed boundary conditions\/}.
        \newblock J. Differential Equations \textbf{265} (2018), no.~4, 1279--1323.

		\bibitem{ISEM}
		\textsc{M.~Egert}, \textsc{R.~Haller}, \textsc{S.~Monniaux}, and \textsc{P.~Tolksdorf}..
		\newblock{Harmonic Analysis Techniques for Elliptic Operators\/}.
		Lecture notes. Available online \url{https://www.mathematik.tu-darmstadt.de/media/analysis/lehrmaterial_anapde/ISem_complete_lecture_notes.pdf}.

		\bibitem{Darmstadt-KatoMixedBoundary}
		\textsc{M.~Egert}, \textsc{R.~{Haller-Dintelmann}}, and \textsc{P.~Tolksdorf}.
		\newblock {\em The {K}ato {S}quare {R}oot {P}roblem for mixed boundary
			conditions\/}. J. Funct. Anal. \textbf{267} (2014), no.~5, 1419--1461.

		\bibitem{Laplace-Extrapolation}
		\textsc{M.~Egert}, \textsc{R.~{Haller-Dintelmann}}, and \textsc{P.~Tolksdorf}.
		\newblock {\em The {K}ato {S}quare {R}oot {P}roblem follows from an extrapolation property of the {L}aplacian\/}. Publ.\ Math. \textbf{61} (2016), no.~2, 451--483.

        \bibitem{Fackler}
        \textsc{S.~Fackler}.
        \newblock {\em Nonautonomous maximal {$L^p$}-regularity under fractional {S}obolev regularity in time\/}.
        \newblock Anal. PDE \textbf{11} (2018), no.~5, 1143--1169.

		\bibitem{GT}
		\textsc{D.~Gilbarg} and \textsc{N.~Trudinger}.
		\newblock Elliptic partial differential equations of second order.
		\newblock Springer, Berlin, 2001.

        \bibitem{Luca_Patrick_Kato}
        \textsc{L.~Haardt} and \textsc{P.~Tolksdorf}.
        \newblock {\em On Kato's Square Root Property for the Generalized Stokes Operator\/}.
        \newblock ArXiv preprint: \url{https://arxiv.org/abs/2410.18787}

		\bibitem{Haase}
		\textsc{M.~Haase}.
		\newblock The {F}unctional {C}alculus for {S}ectorial {O}perators. Operator Theory: Advances and Applications, vol.~169,
		\newblock Birkh{\"a}user, Basel, 2006.

        \bibitem{RobertJDE}
        \textsc{R.~Haller-Dintelmann} and \textsc{J.~Rehberg}.
        \newblock{\em Maximal parabolic regularity for divergence operators including mixed boundary conditions\/}.
        \newblock J. Differential Equations \textbf{247} (2009), no.~5, 1354--1396.

        \bibitem{DeltaArg}
        \textsc{S.~Hofmann}, \textsc{D.~Mitrea}, \textsc{M.~Mitrea}, and \textsc{A.~Morris}.
        \newblock {\em $L^p$-{S}quare {F}unction {E}stimates on {S}paces of {H}omogeneous {T}ype and on Uniformly {R}ectifiable {S}ets\/}.
        \newblock Available online \url{https://arxiv.org/abs/1301.4943}

        \bibitem{Jones_ExtensionOperator}
        \textsc{P.~Jones}.
        \newblock {\em Quasiconformal mappings and extendability of functions in {S}obolev spaces\/}.
        \newblock Acta Math. \textbf{147} (1981), no.~1/2, 71--88.

        \bibitem{JW}
        \textsc{A.~Jonsson} and \textsc{H.~Wallin}.
        \newblock {\em Function spaces on subsets of {${\R}^n$}\/}.
        \newblock Math. Rep. \textbf{2} (1984), no.~1.

		\bibitem{Morris}
		\textsc{A.~Morris}.
		\newblock {\em The {K}ato {S}quare {R}oot {P}roblem on submanifolds\/}.
		\newblock J. Lond. Math. Soc. \textbf{86} (2011), no.~3, 879--910.

        \bibitem{Ouhabaz}
		\textsc{E.-M.~Ouhabaz}.
		\newblock Analysis of {H}eat {E}quations on {D}omains\/.
		\newblock London Mathematical Society Monographs Series, vol.~31,
		Princeton University Press, Princeton NJ, 2005.

        \bibitem{Parabolic_Ouhabaz}
		\textsc{E.-M.~Ouhabaz}.
		\newblock {\em The square root of a parabolic operator\/}.
		\newblock J. Fourier Anal. Appl. \textbf{27} (2021), no.~3.

		\bibitem{Triebel}
		\textsc{H.~Triebel}.
		\newblock Interpolation {T}heory, {F}unction {S}paces, {D}ifferential
		{O}perators.  North-Holland Mathematical Library, vol.~18,
		\newblock North-Holland Publishing, Amsterdam, 1978.


        \bibitem{Triebel-Wavelets}
        \textsc{H.~Triebel}.
        \newblock Function spaces and wavelets on domains.  EMS Tracts in Mathematics, vol.~17,
        \newblock European Mathematical Society (EMS), Z\"{u}rich, 2008.

	\end{thebibliography}
\end{document}